\documentclass[11pt]{amsart}
\usepackage[margin=1.19in]{geometry}
\usepackage{amsmath,amsthm,amssymb,mathtools,amsfonts,mathrsfs}
\usepackage{xspace,xcolor}
\usepackage[breaklinks,colorlinks,citecolor=teal,linkcolor=teal,urlcolor=teal,pagebackref,hyperindex]{hyperref}
\usepackage[alphabetic]{amsrefs}
\usepackage{tikz-cd}
\usepackage[all]{xy}
\usepackage{relsize}
\usepackage{datetime}

\newtheorem{thmx}{Theorem}[section]

\newtheorem{thm}{Theorem}[section]
\newtheorem{prop}[thm]{Proposition}

\newtheorem{lem}[thm]{Lemma}
\newtheorem{cor}[thm]{Corollary}

\theoremstyle{definition}
\newtheorem{defn}[thm]{Definition}
\newtheorem{ex}[thm]{Example}
\newtheorem{rmk}[thm]{Remark}

\newtheorem*{assu*}{Assumption}
\newtheorem*{conv*}{Convention}
\newtheorem{conv}{Convention}
\newtheorem*{fact*}{Fact}
\newtheorem*{ind*}{Induction hypothesis}
\newtheorem*{goal*}{Goal}

\newcommand{\C}{\mathbb C}

\newcommand{\Z}{\mathbb Z}
\newcommand{\N}{\mathbb{N}}

\newcommand{\Pp}{\mathbb P}
\newcommand{\X}{\mathfrak X}
\newcommand{\E}{\mathcal E}
\newcommand{\lin}{\mathfrak l}

\newcommand{\sO}{\mathcal O}
\newcommand{\NE}{\text{NE}}

\newcommand{\bM}{\overline{\mathcal M}}

\newcommand{\comp}{\mathfrak c}
\newcommand{\D}{P}
\author{H.~Fan}
\email{<honglu.fan@math.ethz.ch>}

\author{Y.-P.~Lee}
\email{<yplee@math.utah.edu>}

\email{<ernst.schulteg@t-online.de>}

\date{\today}

\title[Towards QLHT in all genera]
{Towards a quantum Lefschetz hyperplane theorem in all genera}

\begin{document}

\maketitle

\begin{abstract}
An effective algorithm of determining Gromov--Witten invariants of smooth hypersurfaces in any genus (subject to a degree bound \eqref{e:degbound}) from Gromov--Witten invariants of the ambient space is proposed.
\end{abstract}


\section{Introduction}

\subsection{A brief history}
Let $D\subset X$ be a smooth hypersurface in a smooth projective variety. One important question in Gromov--Witten theory is to determine all genus Gromov--Witten invariants of $D$ by the invariants of $X$. An example is the quintic hypersurface in $\Pp^4$. In genus $0$, this relationship enters A.~ Givental's mirror theorem (in for example \cite{Giv4}) as a key step to compute genus $0$ invariants for the quintic 3-folds, or more generally semipositive complete intersections in the projective spaces. This genus zero relationship has been further generalized into the quantum Lefschetz hyperplane theorem in \cites{MR1719555, MR1839288, CG}, relating genus-$0$ Gromov--Witten theory of an ample hypersurface from the (twisted) theory of the ambient space.
One ingredient in this genus zero quantum Lefschetz hyperplane theorem is the functoriality of virtual classes based on a classical idea in enumerative geometry and proven in \cites{Kon, KKP}.

In high genus, a na\"ive generalization of the main result in \cite{KKP} is not valid. That is, the Gromov--Witten theories of hypersurfaces are very different from the twisted theories of the ambient spaces. A simple example is a quartic curve in a quartic surface. Although there is no completely satisfactory proposal of the quantum Lefschetz in higher genus, there have been various approaches, most of which put a special focus on the quintic hypersurface of $\Pp^4$.  In genus $1$, there are works in \cites{Z,LK} among others. In higher genus there is an ongoing \emph{mixed spin $P$-field} approach pioneered by H.-L.~Chang, J.~Li, W.-P.~Li and C.-C.~Liu in, e.g., \cites{CLLL, CGLZ}.  More recently, while this article was under preparation, S.~Guo, F.~Janda and Y.~Ruan in \cite{GJR} announced a new method which determines higher genus invariants of quintic $3$-folds via twisted invariants of the ambient space $\Pp^4$ plus some ``effective invariants'' which are finite for a fixed genus. In the same paper the authors used this in genus $2$ to prove the BCOV holomorphic anomaly conjecture \cite{BCOV2}, and we expect it to be useful in higher genus. There are other interesting works on quintic $3$-folds including \cites{GR1,GR2} and many others. In general, the question remains: To what degree can the Gromov--Witten invariants of a hypersurface of all genera be determined by those of the ambient space?

In this paper, we would like to propose a different method of computing the Gromov--Witten invariants of a hypersurface. Although our current algorithm is subjected to a degree bound \eqref{e:degbound}, similar bounds also explicitly or inexplicitly appear in other works. For example, \cite[Conjecture 1.5]{W} proposed exactly the same bound in determining quintic invariants. Meanwhile, in the work of \cite{GJR}, $2g-2\geq 5d$ appears as a condition of their effective invariants which are needed in computing the quintic invariants. We hope that this different approach provides a new angle and might be of independent interest in Gromov--Witten theory.

\subsection{Weak form of the quantum Lefschetz hyperplane theorem}

Let $D\subset X$ be a smooth hypersurface in a smooth projective variety $X$. Let $\ell \in \operatorname{NE}(D)$ be an effective curve class in $D$. We are only concerned about the cohomology insertions that are restrictions of classes in $X$. We make the following definition.
\begin{defn}\label{defn:nonprim}
Consider the morphism $H^*(X)\rightarrow H^*(D)$ given by the restriction under the inclusion of $D$ into $X$. Denote its image as the following.
\[ H^*(D)_X=im(H^*(X)\rightarrow H^*(D)).\]
\end{defn}
We also make the following convention, which is consistent with all previous results on the quantum Lefschetz hyperplane theorem. 
\begin{conv}\label{conv:nonprim}
When \emph{GW invariants of $D$} is mentioned, we always refer to only invariants whose cohomology insertions lie in $H^*(D)_X$.
\end{conv}
Now we are ready to present some consequences of the results in Section~\ref{s:1.2}.

\begin{thm} \label{t:main}
When $D$ is an ample hypersurface in $X$, the (descendant) Gromov--Witten invariants of $D$ for a fixed genus $g_0$ (of any degree) are determined by GW invariants of $X$ and a \emph{finite number} of GW invariants of $D$ with $g \le g_0$ and 
\begin{equation} \label{e:degbound}
(\ell, D) \leq 2g -2.
\end{equation}

When $D$ is an ample Calabi--Yau three-dimensional hypersurface in $X$, all descendant Gromov--Witten invariants of fixed genus $g$ and effective curve class $\ell \in \operatorname{NE}(D)$ 
can be expressed in terms of (descendant) Gromov--Witten invariants of $X$ and a finite number of invariants of $D$
\[
 \langle \, \rangle^D_{g', 0, \ell'}, \quad \text{for} \,  g' \leq g, \, \ell' < \ell, \text{and} \, (\ell', D) \leq 2g' -2,
\]
where $\langle \, \rangle^D_{g', 0, \ell'}$ is the genus $g'$, $0$-pointed GW invariants on $D$ with class $\ell'$.
In particular, when $D =Q$ is the quintic threefold in $\Pp^4$, and when $g\leq 3$, the GW invariants of $D$ are completely determined.
\end{thm}

This can be considered as a weak form of \emph{quantum Lefschetz hyperplane theorem in higher genus}. Previous results are all in genus zero, cf.\ \cite{MR1719555, MR1839288, CG} and genus one \cite{Z}. 
The main technique in the proof of the above theorem is the virtual localization in an auxiliary space $\X$ (introduced in Section~\ref{s:2.2}).

\begin{rmk}
Let us point our a few take-aways from our method.
\begin{enumerate}
\item It confirms that, in principle, the traditional localization method of moduli of stable maps is powerful enough to effectively determine Gromov--Witten invariants of a hypersurface assuming the knowledge of low-degree invariants (finitely many degrees if the hypersurface is ample) for each genus.
\item The localization formula can be naturally reorganized according to Givental's technique (briefly discussed in section \ref{section:genf}).
\end{enumerate}
For example, since localization is also available in the symplectic geometric definition of Gromov--Witten theory, (a) indicates that we have similar relations in the symplectic setting. It might be possible that these results together with some deep symplectic geometric results (e.g., \cite{IP}) can give us more insights in the algebraic setting.
\end{rmk}



\subsection{Acknowledgement}
The results in this paper were first reported by the first author in a \emph{TIMS algebraic geometry seminar} and in the \emph{NCTS Mini-workshop on Algebraic Geometry} in Dec 2016 during our visit to Taipei. 
Thanks are also due to H.~Lho, F.~Janda, R.~Pandharipande, L.~Wu and Z.~Yang for their interest and stimulating discussions.
The first author is supported by SwissMAP and both authors are partially supported by the NSF during the course of this work.

\section{Statements of main results}

\subsection{Twisted Gromov--Witten invariants} \label{s:1.1}
We recall the definitions and set up the notations of the twisted Gromov--Witten invariants. 
Let $Y$ be a smooth projective variety, $E$ be a vector bundle over $Y$, and $\bM_{g,n}(Y, \beta)$ be the moduli stack of stable maps from genus $g$, $n$-pointed curves to $Y$ with class $\beta \in \operatorname{NE}(Y)$ in the Mori cone of curves.
Let 
\[
ft_{n+1}:\bM_{g,n+1}(Y,\beta) \rightarrow \bM_{g,n}(Y,\beta)
\]
be the map forgetting the last marked point. 
It gives rise to a universal family over $\bM_{g,n}(Y,\beta)$. 
Furthermore, the evaluation map of the last marked point
\[
ev_{n+1}:\bM_{g,n+1}(Y,\beta) \rightarrow Y
\]
serves as the universal stable map over $\bM_{g,n}(Y,\beta)$.

Given the setting, there is a $\C^*$-action on $E$ by scaling the fibers. 
Let $\lambda$ be the corresponding equivariant parameter. 
By projectivity of $Y$, there exists a two term complex of vector bundles over $\bM_{g,n}(Y,\beta)$
\[
0\rightarrow E_{g,n,\beta}^0 \rightarrow E_{g,n,\beta}^1\rightarrow 0
\]
representing the element $R(ft_{n+1})_*ev_{n+1}^*E \in D^b(\bM_{g,n}(Y,\beta))$.
Denote by $E_{g,n,\beta}$ the two term complex $[E_{g,n,\beta}^0 \rightarrow E_{g,n,\beta}^1]$, 
on which there is a $\C^*$-action induced from the action on $E$. 
It is easy to see that the equivariant Euler class of $E_{g,n,\beta}$
\[
e_{\C^*}(E_{g,n,\beta}) :=\displaystyle\frac{e_{\C^*}(E_{g,n,\beta}^0)}{e_{\C^*}(E_{g,n,\beta}^1)} \in H^*(\bM_{g,n}(X,\beta))\otimes_\C \C[\lambda,\lambda^{-1}] 
\]
is well defined.

The twisted Gromov--Witten invariants needed for this paper are of the following form:
\[
\langle \psi^{k_1}\alpha_1,\dotsc,\psi^{k_n}\alpha_n \rangle_{g,n,\beta}^{X, E} := \displaystyle\int_{[\bM_{g,n}(X,\beta)]^{vir}} \dfrac{1}{e_{\C^*}(E_{g,n,\beta})} \cup \prod\limits_{i=1}^n \psi_i^{k_i}ev_{i}^*\alpha_i.
\]
It's an element in $\C[\lambda,\lambda^{-1}]$.

\begin{conv}\label{conv:twisted}
We distinguish two types of the equivariant twisted invariants. When the $\C^*$ acts on the fibers on $E$ by ``positive'' scaling, sending a vector $v$ to $\lambda v$, we write $E^+$ for the equivariant bundle and use the notation .
$$\langle \dotsb \rangle_{g,n,\beta}^{X,E^+}$$ 
for the corresponding equivariant twisted invariant.
Similarly, $\langle \dotsb \rangle_{g,n,\beta}^{X,E^-}$ stands for invariants with inverse scaling action $ v\mapsto \lambda^{-1} v$. 
\end{conv}

\subsection{Quantum Lefschetz} \label{s:1.2}
Let $D\subset X$ be a smooth hypersurface in a smooth projective variety $X$. Let 
\[
vdim_D=vdim(\bM_{g,n}(D,\ell))=(1-g)(\dim (D)-3)+(\ell,-K_D)+n.
\]
In $N_1(D), N_1(X)$, we say $\ell \geq \ell'$ if $\ell-\ell'$ is effective.
Here we abuse the notation and use $\ell, \ell'$ as curve classes both in $D$ and in $X$ via pushforward.
By classical Lefschetz hyperplane theorem, when $\dim D >2$, $H_2 (D) = H_2(X)$ and not much information is lost by this identification. Recall the notation in Definition \ref{defn:nonprim}.

\begin{thm}\label{theorem:recursion}
Let $g\geq 0$, $\ell \in N_1(D)$ such that $2g-2<(\ell,D)$. Let $\alpha_1,\dotsc,\alpha_n\in H^*(D)_X$ be cohomology classes and $a_1,\dotsc,a_n\in \N$ be nonnegative integers such that $\sum\limits_{i=1}^n (a_i+deg(\alpha_i))=vdim_D$. There is a formula expressing the Gromov--Witten invariant
\begin{equation*} \label{e:gwi}
\langle \psi^{a_1}\alpha_1,\dotsc,\psi^{a_n}\alpha_n \rangle^D_{g,n,\ell}
\end{equation*}
in terms of the following.
\begin{itemize}
	\item Invariants of the form $\langle \psi^{k_1}\alpha_1',\dotsc,\psi^{k_{n'}}\alpha_{n'}' \rangle_{g',n', \ell'}^{D,\sO(D)^+\oplus\sO^-}$ where $g'\leq g$, $\ell' < \ell$,
	\item Invariants of the form $\langle \psi^{k_1}\alpha_1',\dotsc,\psi^{k_{n'}}\alpha_{n'}' \rangle_{g',n',\ell'}^{X,\sO(-D)^-}$ where $g'\leq g$ and $\ell' \leq \ell$.
	\item At most one invariant of the form $\langle \psi^{k_1}\alpha_1',\dotsc,\psi^{k_{n}}\alpha_{n}' \rangle_{g,n,\ell}^{X,\sO^+}$
\end{itemize}
\end{thm}

The proof of this theorem will occupy the next section. We first explicate some consequences here. Recall we adopt the Convention \ref{conv:nonprim}.

\begin{cor} \label{c:1.2}
When $D$ is an ample hypersurface in $X$, the (descendant) Gromov--Witten invariants of $D$ for a fixed genus $g_0$ (of any degree) are determined by GW invariants of $X$ and a \emph{finite number} of GW invariants of $D$ with $g \le g_0$ and $(\ell, D) \leq 2g -2$.
\end{cor}

\begin{proof}
By Theorem~\ref{theorem:recursion}, if $(\ell, D) > 2g-2$ it can be expressed in terms of the twisted invariants on $X$ and the twisted invariants on $D$ of lower inductive order.
By the quantum Riemann--Roch theorem in \cite{CG}, the twisted invariants can be expressed in terms of ordinary invariants of the same of lower inductive order.
At the end of the recursive process, the genus and degree bounds will be reached.
Note that any insertion of $1$ can be removed by the string equation, and we assume henceforth no such insertion.
Now, given $g$ and $\ell$, the virtual dimension $vdim_D -n$ above is fixed. 
Therefore, there are only finite number of choices of insertions for a given $n$.
Indeed, when $n >\!> 0$, most of the insertions must be of $\deg_{\mathbb{C}} =1$, which can be removed by divisor and dilaton equations.
Therefore, only a finite number of GW invariants are needed.
\end{proof}

This can be considered as a weak form of \emph{quantum Lefschetz hyperplane theorem in higher genus}. 

When $D$ is a Calabi--Yau threefold, the relation between twisted and untwisted invariants in $D$ is especially simple.

\begin{lem} \label{l:twistedQ}
Suppose $D$ is a Calabi--Yau threefold, the twisted invariants 
\[
\langle \psi^{a_1}\sigma_1,\dotsc,\psi^{a_n}\sigma_n \rangle_{g,n,\ell}^{D,\sO^- \oplus \sO(D)^+}
\]
can be determined by the invariant $\langle \, \rangle^{D}_{g,0,\ell}$ 
\end{lem}

\begin{proof}
Since the virtual dimension of $\bM_{g,n}(D,\ell)$ is $n$, it follows from the string equation, dilaton equation, divisor equation
and the projection formula that
\[
\langle \psi^{a_1}\sigma_1,\dotsc,\psi^{a_n}\sigma_n \rangle_{g,n,\ell}^{D,\sO^-\oplus\sO(D)^+} = C_{g,n,\ell} (a, \sigma)
\langle \, \rangle_{g,0,\ell}^{D,\sO^-\oplus\sO(D)^+},
\]
where $C_{g,n,\ell} (a, \sigma)$ is a constant depending on $g, n, \ell$, $a_i$ and $\sigma_i$ via the above equations. The explicit formula is given in Section~\ref{s:1.3}.

It is not difficult to see that 
\[
 \left( e_{\C^*}((\sO\oplus\sO(D))_{g,0,\ell}) \right)^{-1} =(-1)^{1-g}\lambda^{2g-2-(\ell, D)}(1+O(\lambda^{-1})).
\]
Since $\text{dim}([\bM_{g,0}(D, \ell)]^{vir})=0$ already, the only summand that contributes nontrivially to the integral is the leading term $(-1)^{1-g}\lambda^{2g-2-(\ell, D)}$. As a result,
\begin{equation}\label{eqn:twQ}
\langle \, \rangle_{g,0, \ell}^{D,\sO^-\oplus\sO(D)^+} = (-1)^{1-g} \langle \, \rangle_{g,0,\ell}^D \lambda^{2g-2-(\ell, D)}.
\end{equation}
This completes the proof.
\end{proof}

Note that the constant $C_{g,n,\ell}$ can be made explicit. We present it in the Section~\ref{s:1.3}.
In this (CY3) case, Corollary~\ref{c:1.2} has an especially simple form.

\begin{cor} \label{c:1.4}
When $D$ is an ample CY3 hypersurface in $X$, all descendant Gromov--Witten invariants of fixed $(g, n, \ell)$ with $2g-2 < (\ell, D)$ in $D$ can be expressed in terms of GW invariants of $X$ and a finite number of GW invariants of $D$
\[
 \langle \, \rangle^D_{g', 0, \ell'}, \quad \text{for} \,  g' \leq g, \, \ell' < \ell, \text{and} \, (\ell', D) \leq 2g' -2.
\]
In particular, when $D =Q$ is the quintic threefold in $\Pp^4$, and when $g\leq 3$, the GW invariants of $D$ are completely determined.
\end{cor}

\begin{proof}
The first part follows directly from Lemma~\ref{l:twistedQ}.
For the quintic case, $(\ell, D) \ge 5$ for $\ell \ne 0$.
and hence $2g-2 < (\ell, D)$ for $g \le 3$.
Therefore, the GW invariants of $D$ are completely determined by those of $\Pp^4$,
which in turn are completely determined.
For example, by results of Givental in \cite{MR1901075}, all higher genus invariants of $\Pp^4$ are determined by genus zero invariants.
In genus zero, the $J$-function of $\Pp^4$ are known in \cite{Giv4} and a reconstruction theorem in \cite{LP} completely determines all genus zero (decendant) invariants from the $J$-function.
\end{proof}

\begin{rmk}
Corollaries~\ref{c:1.2} and \ref{c:1.4} are reminiscent of the holomorphic anamoly conjecture of Bershadsky--Cecotti--Ooguri--Vafa \cites{BCOV1, BCOV2}.
The first step towards linking these two will be a generating function formulation alluded in Remark~\ref{r:0.2}(b)  
Indeed, the recent remarkable result by Guo, Janda and Ruan \cite{GJR}, posted while this paper was in preparation, accomplished genus two case via a different method.
\end{rmk}

The main ingredient in formulating and proving Theorem~\ref{theorem:recursion} is the virtual localization on the moduli of stable maps to an auxiliary $(\dim (X)+1)$-dimensional space $\X$, defined in Section~\ref{s:2.2}. 

\subsection{Multiple point functions of Calabi--Yau threefolds} \label{s:1.3}
Now let $D$ be any Calabi--Yau threefold.
As mentioned above, the virtual dimension of $\bM_{g,n}(D,\ell)$ is $n$.
It follows that the cohomology class insertions must be divisors.
In this case, there is a closed formula for the multiple point function in terms of $0$-pointed invariants. 
Let $H$ be any divisor in $D$.
(In the context of GW theory, one often has $H= \sum_i t^i h_i$, where $\{ h_i \}$ is a basis of $H^2(D)$ and $\{ t^i \}$ the dual coordinates.)

\begin{prop}\label{prop:Qvertex}
Let $D$ be a Calabi--Yau threefold.
Fix integers $N, n,m\geq 0$ such that $N\geq n+m$. Let 
\[
\begin{split}
&\sum\limits_{\substack { k_1,\dotsc,k_m, \\ l_1,\dotsc,l_n } } \langle H\psi^{k_1},\dotsc,H\psi^{k_m},\psi^{l_1},\dotsc,\psi^{l_n},1,\dotsc,1 \rangle_{g,N,d\ell}^D \prod_{i=1}^m x_i^{k_i} \prod_{j=1}^n y_j^{l_j} \\
= &\sum\limits_{p=0}^n \left( \prod\limits_{r=1}^p (2g+ m+r-4) \right) \sigma_p(y_1,\dotsc,y_n) (\sum\limits_{i=1}^m x_i + \sum\limits_{j=1}^n y_j)^{N-m-p} \left( \int_{\ell} H \right)^m N_{g,\ell}
\end{split}
\]
where
\[
\sigma_p(y_1,\dotsc,y_n)=\sum\limits_{1\leq j_1< \dotsb < j_p\leq n} y_{j_1}\dotsb y_{j_p}
\]
is the $p$-th elementary symmetric polynomial, $\int_{\ell} H = (\ell, H)$ the pairing, and 
$$N_{g,\ell} :=\langle \, \rangle_{g,0,\ell}^D.$$
\end{prop}

\begin{proof}
One can check that both sides satisfy the string, dilaton and divisor equations, which then uniquely determines the function up to $N_{g,\ell}$.
\end{proof}

The above formula can be used to organize invariants of the following form
\[
\langle \dfrac{A+BH}{w_1-\psi}, \dotsc, \dfrac{A+BH}{w_n-\psi} \rangle_{g,n,\ell}
\]
which occur frequently in the virtual localization.

\begin{cor}\label{cor:Qvertex}
\[
\begin{split}
  &\langle \, \frac{t_0 1+ t_1 H}{1 - z_1\psi}, \dotsc, \frac{t_0 1+t_1 H}{1- z_n \psi} \, \rangle^D_{g,n,d} \\
= &N_{g,d} \sum_{p+q+m=n} t_0^{p+q} \left(\int_{\ell} H \right)^m t_1^m \frac{(2g+m+p-4)!}{(2g+m-4)!} \, \cdot \,   
  \frac{(m+q)!}{m! q!} \sigma_p(z_1,\dotsc,z_n) \left( \sum_{i=1}^n z_i \right)^q ,
\end{split}
\]
where $m, p, q$ are nonnegative integers in the above formula, and by definition
\[
 \frac{(2g+m+p-4)!}{(2g+m-4)!} =  \left( \prod\limits_{r=1}^p (2g+ m+r-4) \right).
\]
\end{cor}

\section{Virtual localization on the master space}\label{section:localization}
In this section, Theorem~\ref{theorem:recursion} is proved.

\subsection{The master space and its fixed loci} \label{s:2.2}

\subsubsection{Compactified deformation to the normal cone} \label{s:2.2.1}
Define $\X$ to be the \emph{compactified deformation to the normal cone} as follows. Consider the product $X\times \Pp^1$. Pick two distinct points on $\Pp^1$ and call them $0$ and $\infty$, then
\[
\X=Bl_{D\times \{0\}}X\times \Pp^1.
\]
There is a birational morphism $\X\rightarrow X\times \Pp^1$, and it can be composed with projections to $X$ and $\Pp^1$. 
We denote the first composition by $p:\X\rightarrow X$ and the second by $\pi:\X\rightarrow \Pp^1$. 
The fiber 
\[
\X_0 := \pi^{-1}(\{0\}) \cong X \cup_D \Pp_{D}(\sO\oplus\sO(D))
\] 
is the union of $X$ and $\Pp_{D}(\sO\oplus\sO(D))$. These two pieces glue transversally along the hypersurface $D \subset X$ and the section
\[
D \cong \Pp_{D}(\sO(D))\subset \Pp_{D}(\sO\oplus\sO(D)).
\]

The following subvarieties of $\X$  will be used frequently.
\begin{itemize}
	\item $X_\infty :=\pi^{-1}(\{\infty\}) \cong X$;
	\item $X_0$ is the irreducible component of $\X_0$ which is isomorphic to $X$;
	\item $D_0 :=\Pp_D(\sO)\subset \Pp_{D}(\sO\oplus\sO(D)) \subset \X_0$.
\end{itemize}

\subsubsection{Fixed loci on $\X$}\label{s:fixedloci}
One can put $\C^*$ actions on the base $\Pp^1$ fixing $0$ and $\infty$. There are different choices and we need to fix one throughout this paper. We make the following convention.
\begin{conv}
The $\C^*$ action acts on the tangent space of $0\in \Pp^1$ with weight $-1$. 
\end{conv}

It induces a $\C^*$ action on $X\times \Pp^1$ by acting trivially on the first factor $X$. 
Since $\X$ is the blow-up of a fixed locus, \emph{there is an induced $\C^*$ action on $\X$}. 
This action acts trivially on $X$, but scales on the fibers of the $\Pp^1$ fibration $\Pp_{D}(\sO\oplus\sO(D))$.

Under this $\C^*$ action, the fixed loci are
\[
 \text{(a)} \, D_0, \quad \text{(b)} \, X_0, \quad \text{(c)} \, X_{\infty} ,
\]
and their normal bundles are
\begin{enumerate}
	\item $N_{D_0/\X}=\sO_D(D)\oplus \sO_D$ with the induced $\C^*$ action of character $1$ on $\sO_D(D)$ factor and of character $-1$ on $\sO_D$ factor.
	\item $N_{X_0/\X}=\sO_{X}(-D)$ with the induced $\C^*$ action of character $-1$ on the fibers.
	\item $N_{X_\infty/\X}=\sO_{X}$ with the induced $\C^*$ action of character $1$ on the fibers.
\end{enumerate}

We introduce the following notations for curve classes.

\begin{defn} \label{d:2.1}
Let $\gamma', \gamma \in \NE(\X)$ such that
\begin{itemize}
	\item $\gamma$ denotes the pushforward of the the fiber class in the $\Pp^1$ bundle $\Pp_D(\sO\oplus\sO(D))$;
	\item $\gamma'$ denotes the class of the strict transform of $\{p\} \times \Pp^1 \subset X \times \Pp^1$ where $p\not\in D$.
\end{itemize}
\end{defn}

\begin{defn}
\begin{itemize}
    \item Let $i_{X_0}:X_0\rightarrow \X$ be the inclusion of $X_0$.
    \item Let $i_{D_0}:D_0\rightarrow \X$ be the inclusion of $D_0$.
\end{itemize}
\end{defn}

\begin{lem}
\[
 N_1 (\X)= (i_{X_0})_*N_1(X) \oplus \Z \gamma \oplus \Z \gamma'. 
\]
\end{lem}
The proof is straightforward.

\subsection{The decorated graphs}
We recall the general framework of virtual localization.
Let $Y$ be a smooth projective variety admitting an action by a torus $T=(\C^*)^m$. It induces an action of $T$ on $\bM_{g,n}(Y,\beta)$ and on its perfect obstruction theory. 
Let $\bM_\alpha$ be the connected components of the fixed loci $\bM_{g,n}(Y,\beta)^T$ labeled by $\alpha$ with the inclusion $i_\alpha: \bM_\alpha \hookrightarrow \bM_{g,n}(Y,\beta)$. 
The virtual fundamental class $[\bM_{g,n}(Y,\beta)]^{vir}$ can be written as
\[
[\bM_{g,n}(Y,\beta)]^{vir}=\sum_\alpha (i_\alpha)_*  \displaystyle\frac{[\bM_\alpha]^{vir}}{e_T(N^{vir}_\alpha)}
\]
where $[\bM_\alpha]^{vir}$ is constructed from the fixed part of the restriction of the perfect obstruction theory of $\bM_{g,n}(Y,\beta)$, and the virtual normal bundle $N^{vir}_\alpha$ is the moving part of the two term complex in the perfect obstruction theory of $\bM_{g,n}(Y,\beta)$ restricted to $\bM_\alpha$.

In the following we apply the above localization to the case $Y=\X$. Moreover, we make the following convention.
\begin{conv}\label{conv:4}
We only consider $\bM_{g,n}(\X,\beta)$ such that $\beta \in (i_{X_0})_*N_1(X) \oplus \N \gamma$ throughout the rest of the paper, with the only exception in Section~\ref{s:2.5}.
\end{conv}

Consider $T=\C^*$ acting on $X$ as described in Section~\ref{s:2.2.1}. 
By analyzing the summands of the virtual localization formula, one can index the fixed loci of $\bM_{g,n}(\X, \beta)$ by the \emph{decorated graphs} defined below.

\begin{defn}\label{decoratedgraph} A \emph{decorated graph} $\Gamma=(\Gamma,\vec{p},\vec{\beta},\vec{s},\vec{\chi})$ for a genus-$g$, $n$-pointed, degree $\beta$ $\C^*$-invariant stable map consists of the following data.

\begin{itemize} 
\item $\Gamma$ a finite connected graph, $V(\Gamma)$ the set of vertices and $E(\Gamma)$ the set of edges;
\item $F(\Gamma)=\{(e,v)\in E(\Gamma)\times V(\Gamma)~|~v~\text{incident to}~ e\}$ the set of flags;
\item the label map $\vec{p}:V(\Gamma)\rightarrow \{D_0,X_0,X_\infty \}$;
\item the degree map $\vec{\beta}:E(\Gamma)\cup V(\Gamma)\rightarrow \NE(\X)$;
\item the marking map $\vec{s}:\{1,2,\dotsc,n\}\rightarrow V(\Gamma)$ for $n>0$;
\item the genus map $\vec{g}:V(\Gamma)\rightarrow \Z_{\geq 0}$.
\end{itemize}
They are required to satisfy the following conditions:
\begin{itemize}
\item $V(\Gamma), E(\Gamma), F(\Gamma)$ determine a connected graph. 
\item $\sum\limits_{e\in E(\Gamma)}\vec{\beta}(e)+\sum\limits_{v\in V(\Gamma)}\vec{\beta}(v)=\beta$.
\item $\sum\limits_{v\in V(\Gamma)} \vec{g}(v) + h^1(| \Gamma |) = g$, where $h^1(|\Gamma|)$ is the ``number of loops'' of the graph $\Gamma$. 
\end{itemize}
To simplify the notations, we will sometimes denote $\vec{p}(v), \vec{\beta}(v), \vec{\beta}(e), \vec{s}(i), \vec{g}(v)$ by $p_v, \beta_v, \beta_e, s_i, g_v$, respectively.
\end{defn}


Let  $f:(C,x_1,\dotsc,x_n)\rightarrow \X$ be a $\C^*$ invariant stable map.
We can associate a decorated graph $\Gamma$ to $f$ as the following. 

\medskip\noindent {\bf Vertices}:
\begin{itemize}
\item The connected components in $f^{-1}(\X^T)$ are either curves or points. Assign a vertex $v$ to a connected component $\comp_v$ in $f^{-1}(\X^T)$.
\item Define $p_v=D_0$, $X_0$ or $X_\infty$ depending on whether $f(\comp_v)\subset D_0$, $X_0$ or $X_\infty$, respectively.
\item When $\comp_v$ is a curve, define $\beta_v=f_*[\comp_v]\in N_1(\X)$. When $\comp_v$ is a point, define $\beta_v=0$.
\item When $\comp_v$ is a curve, define $g_v$ to be the genus of $\comp_v$. When it is a point, define $g_v=0$.
\item When the $i$-th marking lies in the component $\comp_v$, we define $s_i=v$. 
\end{itemize}
\noindent {\bf Edges}:
\begin{itemize}
\item Assign each component of $C-\mathop{\bigcup}\limits_{v\in V(\Gamma)}\comp_v$ an edge $e$. Let $\comp_e$ be the closure of the corresponding component.
\item Write $\beta_e=f_*[\comp_e]\in N_1(\X)$. Due to Convention \ref{conv:4}, $\beta_e$ is a multiple of $\gamma$. Write $k_e$ the integer such that $\beta_e=k_e\gamma$.
\end{itemize}

\begin{rmk}
If the numerical class of $f_*([C])$ had nontrivial coefficient on $\ell'$, the assignment of the graph would have involved a balancing condition on the nodes. See \cite{FL}, or \cite{MM} for details. This general case is not needed in this paper.
\end{rmk}

\begin{defn}
For our convenience, we introduce the following notations.
\begin{itemize}
\item $E_v$ denotes the set of edges that are incident to $v$;
\item $val(v)=|E_v|$ denotes the valence of $v$.
\item $n_v$ is the number of markings on the vertex $v$.
\end{itemize}
\end{defn}

\begin{lem}\label{lem:exc}
Suppose $f_*([C])=\ell\in (i_{D_0})_*N_1(D)$. 
If there is one $v\in V(\Gamma)$ such that $p_v=X_\infty$, then $\Gamma$ is a graph with a single vertex and without any edge. The same conclusion holds if there is one $i$ such that $p_{s_i}=X_\infty$.  In particular, each graph we consider in the paper is either a trivial one over $X_\infty$, or a bipartite graph with vertex components over $D_0$ and $X_0$.
\end{lem}

\begin{proof}
This is due to the fact that the curve class $\beta_\Gamma$ has no horizontal $l'$ components and is connected.
\end{proof}

The fixed loci of $\bM_{g,n}(\X,\beta)$ can be grouped by the decorated graphs. 
Denote by $\bM_\Gamma$ the union of fixed components parametrizing stable maps corresponding to $\Gamma$ and by $N_\Gamma^{vir}$ the virtual normal bundle of $\bM_\Gamma$. 
The next goal is to make the localization residue for each $\bM_\Gamma$ more explicit. A few more definitions are introduced below, partially following \cite[Definition~53]{liu}.

\begin{defn}\label{defn:stable}
A vertex $v\in V(\Gamma)$ is called \emph{stable} if $2g_v-2+val(v)+n_v>0$. Let $V^S(\Gamma)$ be the set of stable vertices in $V(\Gamma)$. Let
\begin{align*}
V^1(\Gamma) &= \{ v\in V(\Gamma) \,|\, g_v=0, val(v)=1, n_v=0 \}, \\
V^{1,1}(\Gamma) &= \{ v\in V(\Gamma) \,|\, g_v=0, val(v)=n_v=1 \}, \\
V^{2}(\Gamma) &= \{ v\in V(\Gamma) \,|\, g_v=0, val(v)=2, n_v=0 \}.
\end{align*}
The union of $V^1(\Gamma), V^{1,1}(\Gamma), V^2(\Gamma)$ is the set of \emph{unstable} vertices.
\end{defn}

\begin{defn}
Define an equivalence relation $\sim$ on the set $E(\Gamma)$ by setting $e_1\sim e_2$ if there is a $v\in V^2(\Gamma)$ such that $e_1, e_2\in E_v$.
Let $\overline E(\Gamma) :=E/\sim$.
\end{defn}

One easily sees that a class $[e]\in \overline E(\Gamma)$ consists of a chain of edges, say $e_1,e_2,\dotsc,e_m$ such that $e_i$ and $e_{i+1}$ intersect at a $v_i\in V^2(\Gamma)$. There are also two vertices $v_0\in e_1$ and $v_m\in e_m$ such that $v_0, v_m\not\in V^2(\Gamma)$.

\begin{defn}\label{defn:vun}
Define $V^{in}_{[e]}=\{v_1,\dotsc,v_{m-1}\}$ and $V^{end}_{[e]}=\{v_0, v_m\}$.
\end{defn}

\begin{defn}\label{defn:leg}
Define $\overline E^{leg}(\Gamma)$ to be the set of edge classes $[e]\in \overline E(\Gamma)$ such that $V^{end}_{[e]}\bigcap V^1(\Gamma)\neq \emptyset$ or $V^{end}_{[e]}\bigcap V^{1,1}(\Gamma)\neq \emptyset$.
\end{defn}

\begin{defn}\label{defn:sides}
$V^{D_0}(\Gamma) :=\{v\in V(\Gamma) \,|\, p_v=D_0\}$ and $V^{P}(\Gamma) :=\{v\in V(\Gamma) \,|\, p_v=X_0\}$.
\end{defn}

\begin{defn}\label{defn:unstablevertex}
Define $V_{[e]}^{D_0,S}=V^{D_0}(\Gamma)\cap (V^{in}_{[e]}\cup (V^{end}_{[e]}\cap V^{S}(\Gamma)))$. In other words, they are all the $D_0$ vertices on the chain $[e]$ except for unstable ones at the two ends. Similarly, we define $V_{[e]}^{X_0,S}=V^{X_0}(\Gamma)\cap (V^{in}_{[e]}\cup (V^{end}_{[e]}\cap V^{S}(\Gamma)))$.
\end{defn}

Definitions \ref{defn:stable}-\ref{defn:unstablevertex} are artificially introduced mostly because we want to label certain terms in the virtual localization formula later.

\subsection{Recursion}\label{section:recursion}
In this subsection we derive the recursion formula which yields Theorem~\ref{theorem:recursion}.
The idea 
is to consider equivariant Gromov--Witten invariants on $X$ which on the one hand vanish by dimension reason, 
and on the other hand provide relations among invariants on $D$ and (twisted) invariants on $X$ thanks to the localization formula.
Here we note that each fixed locus in $\X$ is either isomorphic to $D$ or $X$.
Our claim is that, by carefully choosing the ingredients, one can obtain a set of relations which completely determine invariants on $D$ from those on $X$.

Let $\ell\in \NE(D)$ be an effective curve class. Recall that 
\[
vdim_D=vdim(\bM_{g,n}(D,\ell))=(1-g)(\dim (D)-3)+(\ell,-K_D)+n.
\]
We can similarly denote
\[
vdim_\X=vdim(\bM_{g,n}(\X,(i_{D_0})_*\ell))=(1-g)(\dim (\X)-3)+(\ell,-K_D)+(\ell,D)+n.
\]
Let $\alpha_1,\dotsc,\alpha_n\in H^*(D)_X$ be cohomology classes and $a_1,\dotsc,a_n\in \N$ be nonnegative integers such that $\sum\limits_{i=1}^n (a_i+deg(\alpha_i))=vdim_D$. By pulling back from $X$, there are liftings of these cohomology classes $\tilde\alpha_1,\dotsc,\tilde\alpha_n\in H_{\C^*}^*(\X)$ such that $i_{D_0}^*\tilde\alpha_i=\alpha_i$.
Consider the invariant
\begin{equation}\label{eqn:comp}
\langle \psi^{a_1}\tilde\alpha_1,\dotsc,\psi^{a_n}\tilde\alpha_n \rangle^{\X}_{g,n,(i_{D_0})_*\ell}.
\end{equation}
For virtual dimension reason, we have
\begin{lem}\label{lem:vanishing}
If $2(g-1)< (\ell,D)$, $\langle \psi^{a_1}\tilde\alpha_1,\dotsc,\psi^{a_n}\tilde\alpha_n \rangle^{\X}_{g,n,(i_{D_0})_*\ell}=0$.
\end{lem}

The virtual localization formula states that
\begin{align}\label{eqn:vloc}
\langle \psi^{a_1}\tilde\alpha_1,\dotsc,\psi^{a_n}\tilde\alpha_n \rangle^{\X}_{g,n,(i_{D_0})_*\ell} = \sum\limits_{\Gamma} Cont_\Gamma,
\end{align}
where $Cont_\Gamma$ is the localization residue associated with the fixed locus $M_\Gamma$ whose closed points are invariant stable maps assigned corresponding to the graph $\Gamma$. We can separate graphs with $1$ vertex and graphs with more than $1$ vertices.
\[
\langle \psi^{a_1}\tilde\alpha_1,\dotsc,\psi^{a_n}\tilde\alpha_n \rangle^{\X}_{g,n,(i_{D_0})_*\ell} = \sum\limits_{|V(\Gamma)|=1} Cont_\Gamma + \sum\limits_{|V(\Gamma)|>1} Cont_\Gamma.
\]
This formula can be made more explicit as the following.

\begin{align}\label{eqn:loc}
\begin{split}
& \langle \psi^{a_1}\tilde\alpha_1,\dotsc,\psi^{a_n}\tilde\alpha_n \rangle^{\X}_{g,n,(i_{D_0})_*\ell}\\
=& \langle \psi^{a_1}\alpha_1,\dotsc,\psi^{a_n}\alpha_n \rangle^{D,\sO^-\oplus\sO(D)^+}_{g,n,\ell} + \langle \psi^{a_1}\tilde\alpha_1|_{X_\infty},\dotsc,\psi^{a_n}\tilde\alpha_n|_{X_\infty} \rangle^{X,\sO^+}_{g,n,\ell} +\\
& \sum\limits_{|V(\Gamma)|> 1}\displaystyle\frac{1}{Aut(\Gamma)} \left( \mathlarger{\sum}\limits_{\{i_{[e]}\}_{[e]\in \overline E(\Gamma) \setminus \overline E^{leg}(\Gamma)}} \prod\limits_{v\in V^S(\Gamma)} \langle \dotsc \rangle^{p_v,\E_v}_{g_v,val(v),\beta_v} \right),
\end{split}
\end{align}
where the sum ranges over all decorated graphs of genus $g$, degree $\ell$ with $n$ markings.
A few explanations about the notations are in order:
\begin{enumerate}

\item In ``$\langle \dotsc \rangle^{p_v,\E_v}_{g_v,val(v),\beta_v}$", when $p_v=D_0$, the twisting bundle $\E_v=\sO^- \oplus \sO(D)^+$. When $p_v=X_0$, $\E_v=\sO(-D)^-$. $p_v$ won't be $X_\infty$ in this summation by Lemma \ref{lem:exc}.

\item Let $\{T_i\}$ be a $\C$-basis of $H^*(D;\C)$ and $\{T^i\}$ the dual basis. In the summation, each $i_{[e]}$ in $\{i_{[e]}\}_{e\in \overline E(\Gamma)-\overline E^{leg}(\Gamma)}$ determines an element $T_{i_{[e]}}$ in the basis $\{T_i\}$. And we run over all basis for each edge class in $\overline E(\Gamma)-\overline E^{leg}(\Gamma)$. 

\item $\langle \dotsc \rangle^{p_v,\E_v}_{g_v,val(v),\beta_v}$ is formulated according to the following rules. 
For the $i$-th marking, the class $\psi^{a_i}\alpha_i|_{p_{s_i}}$ is inserted into $\langle \dotsc \rangle^{p_{s_i},\E_{s_i}}_{g_{s_i},val(s_i),\beta_{s_i}}$ as long as $s_i\not\in V^{1,1}(\Gamma)$. For each $[e]\in \bar E(\Gamma)$, insertions are added according to the following rules.
\end{enumerate}

\textbf{Notations:}
\begin{itemize}
    \item For any $[e]\in \overline E(\Gamma)$, let $v_+$ and $v_-$ be the two vertices in $V^{end}_e$ (whichever to be $v_+, v_-$ is arbitrary). 
    \item Define $\iota_{+}$ to be the inclusion $\iota_{+}:D\rightarrow X$ if $p_{v_{+}}=X_0$, or the identity map $\iota_{+}:D\rightarrow D$ if $p_{v_{+}}=D_0$. $\iota_{-}$ is defined in the same way according to $p_{v_-}$. 
    \item Let $e_+$($e_-$ resp.) be the edge in the class $[e]$ that contains $v_+$($v_-$ resp.). 
    \item For a vertex $v$, define $\delta_v$ to be $1$ if $p_v$ is $D_0$, and $-1$ if it is $X_0$.
\end{itemize}

\textbf{Rules:}
For each edge class $[e]\in \overline E(\Gamma)$, an insertion is added into $\langle \dotsb \rangle_{g_{v_{+}},val({v_{+}}),\beta_{v_{+}}}^{p_{v_{+}},\E_{v_+}}$ and $\langle \dotsb \rangle_{g_{v_{-}},val({v_{-}}),\beta_{v_{-}}}^{p_{v_{-}},\E_{v_-}}$ factor for each $v_+, v_-$ as long as it is a stable vertex. These insertions are explicitly described below.
\begin{itemize}
\item Suppose one of $v_+$ and $v_-$ is in $V^1(\Gamma)$. Say $v_-\in V^1(\Gamma)$. Then 
\[
(\iota_+)_* \left( \displaystyle\frac{1}{k_{e_-}} \cdot \frac{Edge(\Gamma,[e])}{\delta_{v_+}(\lambda+D)/k_{e_+}-\psi} \right)
\]
is inserted into the summand $\langle \dotsb \rangle_{g_{v_+},val({v_+}),\beta_{v_+}}^{p_{v_+},\E_{v_+}}$.

\item Suppose one of $v_+$ and $v_-$ is in $V^{1,1}(\Gamma)$. Say $v_-\in V^1(\Gamma)$ and the marking on $v_-$ is the $i$-th marking of $\Gamma$. Then 
\[
(\iota_+)_* \left( \displaystyle\frac{ \left( \delta_{v_-}(\lambda+D)/k_{e_-} \right)^{a_i}\alpha_i|_{p_{s_i}}}{\delta_{v_-}(\lambda+D)} \cdot \frac{Edge(\Gamma,[e])}{\delta_{v_+}(\lambda+D)/k_{e_+}-\psi} \right)
\]
is inserted into the summand $\langle \dotsb \rangle_{g_{v_+},val({v_+}),\beta_{v_+}}^{p_{v_+},\E_{v_+}}$.

\item Suppose otherwise. An insertion 
\[
(\iota_+)_* \left( \displaystyle\frac{Edge(\Gamma,[e])T_{i_{[e]}}}{\delta_{v_+}(\lambda+D)/k_{e_+}-\psi} \right)
\]
should be placed in the summand $\langle \dotsb \rangle_{g_{v_+},val({v_+}),\beta_{v_+}}^{p_{v_+},\E_{v_+}}$. In the meantime an insertion 
\[
(\iota_-)_* \left( \displaystyle\frac{T^{i_{[e]}}}{\delta_{v_-}(\lambda+D)/k_{e_-}-\psi} \right)
\]
should be placed in the $\langle \dotsb \rangle_{g_{v_-},val({v_-}),\beta_{v_-}}^{p_{v_-},\E_{v_+}}$ summand. 
\end{itemize}

This $Edge(\Gamma,[e])$ can be computed in our case.
\[
	Edge(\Gamma,[e])=\displaystyle \frac{1}{k_e}\frac{ (-\lambda)^{|V_{[e]}^{D_0,S}|} D^{|V_{[e]}^{X_0,S}|} }{ \prod\limits_{v\in V^{in}_{[e]}}  \left( \sum\limits_{e'\in E_v}\displaystyle\frac{\delta_{v}(\lambda+D)^2}{k_{e'}} \right) } \prod\limits_{e\in [e]}\displaystyle\frac{\prod\limits_{m=1}^{k_e-1} \left( -\lambda+\displaystyle\frac{m}{k_e}(\lambda+D) \right) }{(-1)^{k_e-1}\left[ \prod\limits_{m=1}^{k_e-1}\displaystyle\frac{m}{k_e}(\lambda+D) \right]^2 }.
\]

\subsubsection{Conclusion of proof}
When $|V(\Gamma)| = 1$, the two possible graphs are $p_v=D_0$ and $p_v=X_\infty$ where $v$ is the unique element in $V(\Gamma)$. These two cases correspond to the two summands in the first line of the right hand side of \eqref{eqn:loc} 
\[
\langle \psi^{a_1}\alpha_1,\dotsc,\psi^{a_n}\alpha_n \rangle^{D,\sO^-\oplus\sO(D)^+}_{g,n,\ell} + \langle \psi^{a_1}\tilde\alpha_1|_{X_\infty},\dotsc,\psi^{a_n}\tilde\alpha_n|_{X_\infty} \rangle^{X,\sO^+}_{g,n,\ell}
\]

To obtain the form of the Theorem \ref{theorem:recursion}, we notice that
\[
\langle \psi^{a_1}\alpha_1,\dotsc,\psi^{a_n}\alpha_n \rangle_{g,n,\ell}^{D,\sO^-\oplus\sO(D)^+}=(-1)^{1-g}\langle \psi^{a_1}\alpha_1,\dotsc,\psi^{a_n}\alpha_n \rangle_{g,n,\ell}^D\lambda^{-2+2g-(\ell,D)}.
\]
The reason is the following. As part of the assumption, we already have $\sum\limits_{i=1}^n(a_i+deg(\alpha_i))=vdim_D$. We know the twisting
\[\dfrac{1}{e_{\C^*}(\sO_{g,n,\ell}\oplus\sO(D)_{g,n,\ell})}=\lambda^{-2+2g-(\ell,D)}+\dotsc.\]
But there is no room for the lower degree terms to come into play. In equation \eqref{eqn:loc}, one can easily verify that all other terms are lower order terms as described in Theorem \ref{theorem:recursion}.
\emph{This concludes the proof of Theorem \ref{theorem:recursion}}.

\subsection{An Example}
Let $D$ be $\Pp^1$ and $X$ be $\Pp^2$. $D$ embeds into $X$ as a line. Consider $\ell$ to be the degree $1$ class in $D$. Applying equation \eqref{eqn:loc}, we get the following.
\begin{equation}\label{eqn:ex1}
\langle ~ \rangle_{0,0,1}^D \lambda^{-3} = \langle ~ \rangle_{0,0,1}^{\Pp^2,\sO^+}+ \text{(lower order terms)} .
\end{equation}
Denote the hyperplane class in $\Pp^2$ by $H$. One easily finds out that the lower order terms consists of
\begin{equation}
\displaystyle \langle \frac{H}{-\lambda-\psi-H}  \rangle_{0,1,1}^{\Pp^2,\sO(-1)^-} .
\end{equation}
Expand the geometric series 
\[
\displaystyle\frac{H}{-\lambda-\psi-H} =-\frac{H}{\lambda}\left( 1-\frac{\psi+H}{\lambda} + \left(\frac{\psi+H}{\lambda}\right)^2 + \dotsc  \right).
\]

By counting the virtual dimension, the only nonzero term is
\[
-\langle \frac{H(\psi+H)^2}{\lambda^3}  \rangle_{0,1,1}^{\Pp^2,\sO(-1)^-} .
\]
One expands to get
\[ 
-\langle \frac{H\psi^2+2H^2\psi}{\lambda^3}  \rangle_{0,1,1}^{\Pp^2,\sO(-1)^-} .
\]
Since $h^0(\Pp^1,\sO(-1))=h^1(\Pp^1,\sO(-1))=0$, the above is nothing but an untwisted invariant. One expand to get
\[
-\langle H\psi^2  \rangle_{0,1,1}^{\Pp^2}\lambda^{-3} - 2 \langle H^2\psi  \rangle_{0,1,1}^{\Pp^2}\lambda^{-3} .
\]
Writing out the small $J$-function for $\Pp^2$, one gets
\[
\langle H\psi^2  \rangle_{0,1,1}^{\Pp^2}=-3,\hspace{1in} \langle H^2\psi  \rangle_{0,1,1}^{\Pp^2}=1 .
\]
Plugging everything into Equation (\ref{eqn:ex1}), one computes $\langle ~ \rangle_{0,0,1}^D=1$ which is exactly what we expect.

\section{Towards a generating function formulation} \label{section:genf}
In this section, a generating function formulation of Theorem~\ref{theorem:recursion} in the case of $D=Q$, the quintic hypersurface in $X=\Pp^4$, is sketched.

It should be possible to organize the recusion \eqref{eqn:loc} for general pair of $(X,D)$ into a more useful form using generating functions. 
An interesting and simplifying situation is when $D$ is a Calabi-Yau hypersurface of dimension $3$, when the notations are greatly simplified. 
Here we consider the case when $X=\Pp^4$ and $D=Q$ the quintic hypersurface. It is our hope that by reorganizing our relations in terms of generating functions, the partition function $F_g$ might be expressed in terms of $F_{g'}$ with $g'<g$ (mod finitely many low degree terms). In this section, we would like to demonstrate the generating function form of Equation \eqref{eqn:loc} under this specific quintic 3-fold case. The general case should work similarly. This formulation follows Givental's approach \cite{MR1866444, MR1901075}. (See also \cite{LPbook}.)

\subsection{A sketch of formulation}

We present the sketch of generating function formulation in this subsection. While some ingredients can be computed explicitly here, like the hypertail contribution at $Q_0$, others haven't been done yet. The I-function and the hypertail contribution at $Q_0$ will be computed later in section \ref{section:J} and \ref{section:hypertail}. 
These explicit computations are also used in demonstrating the fact that the degree bound \eqref{e:degbound} cannot be improved, as is discussed in Section~\ref{s:A.3}.

Let us introduce the following notations
\[
\langle\!\langle \alpha_1,\dotsc,\alpha_n \rangle\!\rangle_{g}^Q(t)=\sum\limits_{d>0}\sum\limits_{n=0}^\infty \dfrac{1}{m!}\langle \alpha_1,\dotsc,\alpha_n,t,\dotsc,t \rangle_{g,n+m,d\lin}^Qq^{d\lin},
\]
\[
\langle\!\langle \alpha_1,\dotsc,\alpha_n \rangle\!\rangle_{g}^{X,\sO(-5)^-}(t)=\sum\limits_{d>0}\sum\limits_{n=0}^\infty \dfrac{1}{m!}\langle \alpha_1,\dotsc,\alpha_n,t,\dotsc,t \rangle_{g,n+m,d\lin'}^{X,\sO(-5)^-}q^{d\lin'}.
\]
Let $S^\X(t)$ be a linear operator on $H^*(X)$ defined as follows.
\[
	S^\X(t,z)T=T+\sum\limits_{d_1>0,d_2\geq 0}q^{d_1\lin'+d_2\gamma} \left\langle\!\!\left\langle \dfrac{T}{z-\psi},T^\alpha \right\rangle\!\!\right\rangle_{0,d_1\lin'+d_2\gamma}^XT_\alpha.
\]
The big $J$-function $J^\X(t)$ is defined as 
\[
J^\X(t,z)=1+(S^\X(t,z)(T^\alpha),1)_{\X}T_\alpha,
\]
where $(,)_\X$ is the intersection pairing on $\X$.
Also let $V^\X(t,w,z)$ be a bilinear pairing on $H^*(X)$ as follows.
\[
    V^\X(t,w,z)(T_1,T_2)=\dfrac{(T_1,T_2)}{w+z}+\sum\limits_{d_1>0,d_2\geq 0}q^{d_1\lin'+d_2\gamma} \left\langle\!\!\left\langle \dfrac{T_1}{w-\psi},\dfrac{T_2}{z-\psi} \right\rangle\!\!\right\rangle_{0,d_1\lin'+d_2\gamma}^X.
\]
We have the standard fact that 
\[
V^X(t,w,z)(T_1,T_2)=\dfrac{\bigg(S(t,w)T_1,S(t,z)T_2\bigg)}{w+z}.
\]
We make the following definitions.
\begin{defn}
\begin{enumerate}
    \item $t_{Q_0}=[zJ^\X(t=0,z)-z]_+|_{Q_0}, \qquad t_{X_0}=[zJ^\X(t=0,z)-z]_+|_{X_0}$.
    \item $\hat t_{Q_0}(T)=[S^\X(t=0,z)T]_+|_{Q_0}, \qquad \hat t_{X_0}(T)=[S^\X(t=0,z)T]_+|_{X_0}$.
    \item $E(w,z;T_1,T_2)=[V^X(t=0,w,z)(T_1,T_2)]^{z,w}_+$.
\end{enumerate}
Here $[\dotsb]^{z,w}_+$ means keeping only the $w^az^b$ terms with $a,b\geq 0$ while expanded as $z/\lambda, w/\lambda$ series (the $(T_1,T_2)/(w+z)$ term is also thrown away).
\end{defn}
We call $t_{Q_0}, t_{X_0}$ the \emph{hypertail contributions} at $Q_0$ and $X_0$, respectively.
\begin{rmk}
The $I$-function is calculated in Section~\ref{section:J} and the restriction of \emph{small} $J^{\X}$ to $Q_0$ can be explicitly read off from the $I$-function. $t_{Q_0}$ is therefore calculated in Section~\ref{section:hypertail}.
In principle, all hypertails can be computed. But it is unclear to us whether it can be made explicit.
\end{rmk}

Since localization formula is homogeneous in $\lambda$, we will set $\lambda=1$ for simplicity. Let $\Omega$ be the set of genus $g$ stable graphs, and $\Gamma_0$ the stable graph with one single genus $g$ vertex. To make the text more readable, only in equation \eqref{eqn:generatingfunc} \eqref{eqn:genfuncex}, we abuse the notation by temporarily
\begin{itemize}
    \item write $\langle\!\langle \dotsb \rangle\!\rangle_{g}^Q$ instead of $\langle\!\langle \dotsb \rangle\!\rangle_{g}^Q(t=t_{Q_0}(-\psi))$, and
    \item write $\langle\!\langle \dotsb \rangle\!\rangle_{g}^{X,\sO(-5)^-}$ instead of $\langle\!\langle \dotsb \rangle\!\rangle_{g}^{X,\sO(-5)^-}(t=t_{X_0}(-\psi))$.
\end{itemize}
Now with these definitions, we can apply localization formula to the following generating series.
\begin{align}\label{eqn:generatingfunc}
\begin{split}
    &\sum\limits_{d_1>0,d_2\geq 0} q^{d_1\lin'+d_2\gamma}\langle ~ \rangle_{g,0,d_1\lin'+d_2\gamma} \\
    =&\sum\limits_{d\geq 0} \langle ~ \rangle_{g,0,d\lin}^{X,\sO^-}q^{d\lin}+(-1)^{1-g} \left\langle\!\left\langle ~ \right\rangle\!\right\rangle^Q_{g}+ \left\langle\!\left\langle ~ \right\rangle\!\right\rangle^{X,\sO(-5)^-}_{g}+ \\
    &\sum\limits_{\Gamma\in \Omega \atop{\Gamma\neq \Gamma_0}} \sum\limits_{\{i_e\}_{e\in E(\Gamma)}\atop{\{a_e,b_e\}_{e\in E(\Gamma)}}} \prod\limits_{v\in V(\Omega)}\left( (-1)^{1-g_v}\left\langle\!\left\langle\dotsb\right\rangle\!\right\rangle_{g_v}^Q + \left\langle\!\left\langle\dotsb\right\rangle\!\right\rangle_{g_v}^{X,\sO(-5)^-} \right).
\end{split}
\end{align}
Here the term $\sum\limits_{\{i_e\}_{e\in E(\Gamma)}\atop{\{a_e,b_e\}_{e\in E(\Gamma)}}} \prod\limits_{v\in V(\Omega)}\left( (-1)^{1-g_v}\left\langle\!\left\langle\dotsb\right\rangle\!\right\rangle_{g_v}^Q + \left\langle\!\left\langle\dotsb\right\rangle\!\right\rangle_{g_v}^{X,\sO(-5)^-} \right)$ is defined by packaging the insertion rules in equation \eqref{eqn:loc}. The explanation will be similar to the one in equation \eqref{eqn:loc}. The text will be unnecessarily long but the idea is very straightforward. Therefore, instead of explaining it into every detail, we would like to demonstrate how to write this term in the case when $\Gamma$ consists of two vertices of genus $g_1$ and $g_2$, respectively. Let $\{T_1,\dotsc,T_N\}$ be a basis of $H^*(X)$ and $\{T^1,\dotsc,T^N\}$ its dual basis.

\begin{ex}
Suppose $\Gamma$ consists with two vertices of genus $g_1$ and $g_2$ with an edge connecting them. This last summand can be written as the following.
\begin{align}\label{eqn:genfuncex}
\begin{split}
\sum\limits_{N\geq i\geq 1\atop{ a,b\in\Z_{\geq 0} }} &[E(w,z;T_i,T^i)]_{w^az^b}\left( (-1)^{1-g_v}\left\langle\!\left\langle (-\psi)^aT^i|_{Q_0} \right\rangle\!\right\rangle_{g_v}^Q + \left\langle\!\left\langle (-\psi)^aT^i|_{X_0} \right\rangle\!\right\rangle_{g_v}^{X,\sO(-5)^-} \right) \\
&\left( (-1)^{1-g_v}\left\langle\!\left\langle (-\psi)^bT_i|_{Q_0} \right\rangle\!\right\rangle_{g_v}^Q + \left\langle\!\left\langle (-\psi)^bT_i|_{X_0} \right\rangle\!\right\rangle_{g_v}^{X,\sO(-5)^-} \right)
\end{split}
\end{align}
\end{ex}
If there are more edges, the matching of insertions should be made similarly. An easy implication of our Theorem \ref{theorem:recursion} is that $\sum\limits_{d_1>0,d_2\geq 0} q^{d_1\lin'+d_2\gamma}\langle ~ \rangle_{g,0,d_1\lin'+d_2\gamma}$ do not have those terms where $2-2g+5d_1+d_2>0$. Hence
\[
\bigg[ \sum\limits_{d_1,d_2\geq 0} q^{d_1\lin'+d_2\gamma}\langle ~ \rangle_{g,0,d_1\lin'+d_2\gamma} \bigg]_{2-2g+5d_1+d_2>0}=0,
\]
where $[\dotsb]_{2-2g+5d_1+d_2>0}$ means we truncate the series leaving the $q^{d_1\lin'+d_2\gamma}$ terms such that the inequality holds. 
The truncation can often be ignored in low genus. For example when $g=2,3$, we can simply say
\[
\sum\limits_{d_1>0,d_2\geq 0} q^{d_1\lin'+d_2\gamma}\langle ~ \rangle_{g,0,d_1\lin'+d_2\gamma}=0.
\]
On the other hand, the genus $g$ quintic partition function is already all collected in one of the summands on the right hand side of equation \eqref{eqn:generatingfunc}. By string/dilaton/divisor equations, the summand $\left\langle\!\left\langle ~ \right\rangle\!\right\rangle^Q_{g}(t=t_{Q_0}(-\psi))$ can be reduced into a variation of quintic partition function in the following form.
\begin{equation}
    \left\langle\!\left\langle ~ \right\rangle\!\right\rangle^Q_{g}(t=t_{Q_0}(-\psi))=\sum\limits_{d\geq 0} F_{d,g}(q^{\lin'},q^\gamma)\langle ~ \rangle_{g,0,d}^Q.
\end{equation}
$F_{d,g}(q^{\lin'},q^\gamma)$ is computed in equation \eqref{e:20} in the Appendix and simplified in section \ref{section:genfun}.

We believe equation \eqref{eqn:generatingfunc} could shed some light into describing the quintic partition functions. But it requires more talented people to provide new inputs into this computation, especially regarding how to understand the generating series $\langle\!\langle \dotsb \rangle\!\rangle_g^{\Pp^4,\sO(-5)^-}$.

\subsection{The I-function of $\X$} \label{section:J}
We demonstrate the calculation of the equivariant small $I$-function and $J$-function of $\X$ in this subsection. The big $I$-function (and hence $J$-function in principle) are also computable, but are far more complicated.

Recall $\X$ has a map to $\Pp^4\times \Pp^1$, and the $\C^*$ action is a lifting from an action on $\Pp^4\times \Pp^1$. The $\C^*$ acts on $\Pp^4$ trivially. But a $\lambda\in \C^*$ acts on $\Pp^1$ by sending $[x:y]$ to $[x:\lambda y]$ (Suppose $[0:1]$ is our $0\in \Pp^1$). The equivariant cohomology on $\Pp^1$ has the presentation
\[
	H_{\C^*}^*(\Pp^1)=\C[h,\lambda]/(h-\lambda)h.
\]
Denote $H$ the pullback of the hyperplane class in $\Pp^4$, and $E$ the (equivariant) exceptional divisor in $X$. Notice that one can realize $\X$ as an equivariant hypersurface in $\Pp_{\Pp^4\times\Pp^1}(\sO(-K_{\Pp^4})\oplus \sO(h))$. By a slight abuse of notation, $\sO(-K_{\Pp^4})$ and $\sO(h)$ are pulled-back from $\Pp^4$ and $\Pp^1$, respectively. Under this embedding, equivariant divisor $\X$ is of class $-K_{\Pp^4}+h+\sO_P(1)$. 

Denote $D$ by $Q$, $D_0$ by $Q_0$ and introduce the following notations:
\begin{itemize}
    \item $\lin\in \NE(\X)$ is the line class supporting on $Q_0$;
    \item $\lin'\in \NE(\X)$ is the line class supporting on $X_0$.
\end{itemize}
It is easy to see that the Mori cone is generated by $\lin'$, $\gamma$ and $\gamma'$ (defined in Definition~\ref{d:2.1}) and $\lin = \lin' + 5 \gamma$.
We write 
$$\beta = d_1 \lin' + d_2 \gamma + d_3 \gamma' = d_1 \lin + (d_2 - 5 d_1) \gamma + d_3 \gamma'.$$

It follows from the results of Givental \cite{Giv2} that the small $I$-functions of GW theory of hypersurfaces in smooth toric varieties can be written down explicitly
{\small
\[
\begin{split}
	&I^\X=e^{\D/z}\sum\limits_{\beta} q^{\beta} 
	e^{(\beta, \D)} I_{\beta}, \\
	I_{\beta}=&\dfrac{\prod\limits_{m=1}^{d_2}(-E+5H-\lambda+h+mz) }{\prod\limits_{m=1}^{d_1}(H+mz)^5 \prod\limits_{m=1}^{d_2-5d_1}(-E-\lambda+h+mz) \prod\limits_{m=1}^{d_2-d_3}(-E+5H+mz)
	\prod\limits_{m=1}^{d_3}(h-\lambda+mz)
	\prod\limits_{m=1}^{d_3}(h+mz)},
\end{split}
\]
}where $\D=t_01+t_1H+t_2E+t_3h$ is a general element in $H^{\leq 2}(\X)$. We use the convention that $I_0 =1$, $1/(-n)! =0$ when $n$ is a positive integer, and
\[
 \frac{1}{\prod\limits_{m=1}^{d}(P+mz)} := \frac{\prod\limits_{m=-\infty}^{0}(P+mz)} {\prod\limits_{m=-\infty}^{d}(P+mz)}.
\]

To obtain the $J$-function \cite{Giv2}, we need to perform the mirror transformation. The summand that has nonzero coefficient of $1/z$ is when $d_1=0, d_2=1, d_3=0$ or $d_2=0, d_3=0$ with $d_1$ being arbitrary. If we only restrict to the curve classes $\beta=d_2\gamma$, the mirror transformation can be read off from the $1/z$ coefficient of the following.
\[
	q^{\gamma}e^{\D/z+t_1-t_2}\dfrac{-E+5H-\lambda+h+z}{(-E-\lambda+h+z)(-E+5H+z)}=q^\gamma e^{\D/z+t_1-t_2}\left( 1/z+O(1/z^2) \right).
\]
Therefore the mirror transformation takes $t_0$ to $t_0 + q^\gamma e^{t_1-t_2}$, and we get the restricted $J$-function as follows
\[
	J^\X(d_1=d_3=0)= e^{(-q^\gamma e^{t_1-t_2})/z} I^{\X}(d_1=d_3=0).
\]
This will be made explicit in the next section.

\subsection{The hypertail $t_{Q_0}$} \label{section:hypertail}

In this subsection, we show how to calculate the hypertail contributions and explicitly compute its restriction to Novikov variables $q^{k\gamma}$. 
Firstly, we recall that in the present setting, the $\C^*$ fixed points consist of three components: $Q_0$, $X_0$ and $X_{\infty}$.
We are interested in the hypertail associated to $Q_0$.

Recall the following definition 
$t_{Q_0}=\left[ \left( zJ^\X(t=0,z) \right)|_{Q_0} -z \right]_+ $,
where $[\dotsb]_+$ is the power series trunction in $z$.

A few explanatory remarks are in order.
The divisors restrict to $Q_0$ as follows
$$E|_{Q_0}=-\lambda, \quad E|_{X_0}=5H, \quad h|_{Q_0}=0, \quad h|_{X_0}=0.$$
By a slight abuse of notation, we denote $H |_{Q_0}$ as $H$.
Furthermore, the ``gamma'' factors in the denominator of the $J$-function expand according to the following rules extracted from the localization. When there is no $\lambda$ in the factor, it expands as a power series in $z^{-1}$
\[
 \frac{1}{\prod\limits_{m=1}^{d}(P+mz)} = \left( \frac{1}{d! z^d} \right) \frac{1}{\prod\limits_{m=1}^{d}(1+ \frac{P}{mz})} 
\]
and $(1+ \frac{P}{mz})$ then expands as geometric series in $z^{-1}$.
If it involves $\lambda$, then it expands in power series in $z$ and $\lambda^{-1}$. For example,
\[
 \frac{1}{\prod\limits_{m=1}^{d}(P+ \lambda + mz)} = \frac{1}{\prod\limits_{m=1}^{d} (\lambda+mz)} \frac{1}{\prod\limits_{m=1}^{d}(1+ \frac{P}{\lambda + mz})} 
\]
and $(1+ \frac{1}{\lambda + mz})$ then expands as geometric series in $z$.

We now explain the geometric meaning of the hypertail contribution.
Without loss of generality, we set all $t_i=0$. (The values can be reconstructed from the string and divisor equations.) By \cite{GB} or \cite{FL} the Laurent series expansion in $1/z$ (power of $z$ is bounded below, but unbounded above) of $J^\X |_{Q_0}$, we have
\begin{equation}\label{eqn:Jrestr}
	J^\X
	|_{Q_0}=1+\dfrac{t_{Q_0}(z)}{z}+\sum\limits_{d>0,n}\dfrac{q^{d\lin}}{n!} T_\alpha \left\langle \dfrac{T^{\alpha}}{z(z-\psi)},t_{Q_0}(-\psi),\dotsc,t_{Q_0}(-\psi) \right\rangle_{0,n+1,d \lin}^Q,
\end{equation}
where $t_{Q_0}(z)$ is a series in $z$ and Novikov variable $q^{\lin}$. More precisely, $t_{Q_0}(z)$ is a partial sum of the localization formula for the function
\[
	zJ^\X 
	-z=\sum\limits_{d_1\geq 0,d_2\geq 0}q^{d_1\lin'+d_2\gamma} T_\alpha \left\langle \dfrac{T^{\alpha}}{z-\psi} \right\rangle_{0,1,d_1\lin'+d_2\gamma}^\X,
\]
summing only over the graphs such that
\begin{enumerate}
\item first marking lies on a vertex $v$ over $Q_0$ of degree $0$ genus $0$;
\item there is a single edge coming out of vertex $v$.
\end{enumerate}
Geometrically, $t_{Q_0}$ is a sum of tail contributions emanating from the vertex $Q_0$, and thus named. 

\begin{ex}\label{ex3.4}
Let's take $d_1=0$. 
\begin{equation} \label{e:specialJ}
	J_0 := J^\X(t_i=0,q^\lin=0)|_{Q_0}=e^{-q^\gamma/z}\sum\limits_{d_2\geq 0}q^{d_2\gamma}\dfrac{ \prod\limits_{m=1}^{d_2}(5H+mz) }{ \prod\limits_{m=1}^{d_2}(mz) \prod\limits_{m=1}^{d_2}(5H+\lambda+mz) }.
\end{equation}
One can thus read off the corresponding hypertail contribution by multiplying $z$ and take a truncation for positive power terms. For example when $d_2=1$, the corresponding hypertail contribution is
\[
	\dfrac{5H+z}{5H+\lambda+z}-1=\dfrac{-\lambda}{5H+\lambda+z},
\]
where the denominator is expanded as a $z$-series. When $d_2=2$, the corresponding hypertail contribution is
\[
	\left[  \dfrac{1}{z}\left(\dfrac{1}{2}-\dfrac{5H+z}{(5H+\lambda+z)} + \dfrac{(5H+z)(5H+2z)}{2(5H+\lambda+z)(5H+\lambda+2z)}\right) \right]_+,
\]
where $[~]_+$ is the truncation to get the positive power terms. In this case, one can check by hand that it's the sum of two graphs: one with a degree $2$ edge and one with two consecutive degree $1$ edges.
\end{ex}

\begin{ex}
It is not completely obvious how to handle \emph{directly} genus $0$ degree $1 \lin$ localization computation via summation over graphs due to the complexity of the arrangement of $5$ fiber classes. Using the hypertails to compute it is relatively easy. First, we extract the coefficients of $q^{d_2 \gamma}$ in $t_{Q_0}(z)$ for $d_2 \leq 5$ (via computers). Then, we are left to determine the following expression
\[
\bigg[ \sum\limits_{n\geq 0}\dfrac{1}{n!}\langle t_{Q_0}(-\psi),\dotsc,t_{Q_0}(-\psi) \rangle_{0,n,1}^{\Pp^4,\sO(-5)^-} \bigg]_{q^{5\gamma}}.
\]
It can be easily done by imposing the $(\C^*)^5$ action on $\Pp^4$ and using localization. Our Maple program successfully returns $2875$ as the answer. In principle, one can write a program to compute twisted invariants of $\Pp^4$ and use this method to get other low degree/genus quintic invariants.

\end{ex}

The hypertails $t_{Q_0}$ and $t_{X_0}$ can be computed, and will enter as a component in the generating function formulation later.
Below we write down the case where $d_3=0$ explicitly for $t_{Q_0}$. The general case can be done similarly, but the formulas are more complicated.


\begin{defn}
Define $H_k=\sum\limits_{i=1}^k 1/i$ to be the harmonic series.
\end{defn}

\begin{prop}
Set $q^{\gamma'} =0$ ($d_3 =0$). Then we have
\begin{align}
\begin{split}
\displaystyle t_{Q_0}(q^{\gamma'}=0)=&\sum\limits_{m=1}^\infty \frac{1}{\lambda+mz} \sum\limits_{d_1\geq 0, d_2\geq m} \dfrac{q^{d_1\lin'+d_2\gamma}}{\lambda^{d_2-2}} A_1(d_1,d_2,m) \\
-& H \cdot \sum\limits_{m=1}^\infty \frac{1}{\lambda+mz} \sum\limits_{d_1\geq 0, d_2\geq m} \dfrac{q^{d_1\lin'+d_2\gamma}}{\lambda^{d_2-1}} 5A_2(d_1,d_2,m) \\
-& H \cdot \sum\limits_{m=1}^\infty \frac{1}{(\lambda+mz)^2} \sum\limits_{d_1\geq 0, d_2\geq m} \dfrac{q^{d_1\lin'+d_2\gamma}}{\lambda^{d_2-2}} 5A_1(d_1,d_2,m) 
\,  +O(H^2),
\end{split}
\end{align}
where
\[
A_1(d_1,d_2,m)={5 d_1 \choose d_2 -m} \dfrac{(-1)^{d_2-m-1}m^{d_2-2}}{(d_1 !)^5 (d_2 -5 d_1)!},
\]
\[
A_2(d_1,d_2,m)=\dfrac{(5d_1)!}{(d_1!)^5} (-1)^{d_2-m-1} m^{d_2-1} \left( \frac{C(d_1,d_2,m)}{(m-1)!} + \frac{ (d_2 - 2 - mH_{d_1}) }{m!(d_2-m)!}{m\choose d_2-5d_1} \right),
\]
\[\displaystyle
C(d_1,d_2,m)=\sum\limits_{i=0}^{d_2-m} \dfrac{(-1)^i}{i!} {d_2-i \choose 5d_1}  \dfrac{H_{d_2-i}}{(d_2-i-m)!}.
\]
Here we use the convention that $\displaystyle {a \choose b}=0$ if $b>a$ and $\frac{1}{n!} =0$ when $n$ is negative.
\end{prop}

\begin{proof}
Recall the explicit form of the $I$-function from Section \ref{section:J}. Also recall that its $1/z$ term consists of two parts:
\begin{enumerate}
    \item $q^{\gamma}e^{P/z+t_1-t_2}/z$;
    \item Terms associated to Novikov variables $q^{d_1\lin'}$ (i.e., where $d_2=d_3=0$).
\end{enumerate}
Part (b) appears because of the following hypergeometric factor.
\[
\dfrac{\prod_{m=-\infty}^0(-E-\lambda+h+mz)}{\prod_{m=-\infty}^{-d_1}(-E-\lambda+h+mz)}.
\]
However we are only concerned with the hypertail contribution at $Q_0$. Since $E|_{Q_0}=\lambda, h|_{Q_0}=0$, part (b) vanishes for $I^{\X}|_{Q_0}$, and we can compute the restriction of $J$-function using only (a).

The computation is straightforward. We roughly sketch the methods and leave the details to the readers. It's easy to see that
\begin{align}
\begin{split}
J^\X&|_{Q_0} 
= e^{-q^\gamma/z}\sum\limits_{d_1,d_2\geq 0}q^{d_1\lin'+d_2\gamma}
\dfrac{ \prod\limits_{m=1}^{d_2}(5H+mz) }{ \prod\limits_{m=1}^{d_1}(H+mz)^5 \prod\limits_{m=1}^{d_2-5d_1}(mz) \prod\limits_{m=1}^{d_2}(5H+\lambda+mz) } \\
&= e^{-q^\gamma/z}\sum\limits_{d_1,d_2\geq 0}q^{d_1\lin'+d_2\gamma} \dfrac{ \dfrac{d_2!}{(d_1!)^5(d_2-5d_1)!}\left( 1 + 5H \left( H_{d_2}-H_{d_1}-\sum\limits_{m=1}^{d_2}\dfrac{1}{\lambda+mz} \right) \right) }{\prod\limits_{m=1}^{d_2}(\lambda+mz)}.
\end{split}
\end{align}

We expand $e^{-q^\gamma/z}=\sum\limits_{i=0}^\infty \left(\dfrac{-q^\gamma}{z}\right)^i/i!$. Recall that 
$t_{Q_0}=[zJ^\X(t_i=0)-z]_+$.
Our general strategy is the following. Suppose $A(z)$ is a power series in $z$. Notice
\[\left[\dfrac{1}{z}A(z)\right]_+=\dfrac{A(z)-A(0)}{z}.\]
Therefore, we can simply regard each $1/z$ as such a difference operator. Its effect on a general power series is complicated. But note that the effect of $1/z$ on $\dfrac{1}{\lambda+mz}$ is just a multiplication by scalar as follows. 
\[\left[\dfrac{1}{z} \cdot \dfrac{1}{\lambda+mz}\right]_+=\dfrac{-m}{\lambda} \cdot \dfrac{1}{\lambda+mz}.\]
If we expand the summands into partial fractions, we can easily apply these difference operator and finish our computation. This can be done by Lagrange interpolation. For example,
\begin{align}
\begin{split}
    \dfrac{1}{\prod\limits_{m=1}^{d_2}(\lambda+mz)} &= \sum\limits_{m=1}^{d_2} \dfrac{1}{\lambda+mz} \cdot \dfrac{1}{\prod\limits_{j\neq m}(\lambda-j\lambda/m)} \\
    &= \sum\limits_{m=1}^{d_2} \dfrac{1}{\lambda+mz} \cdot \dfrac{(-1)^{d_2-m}m^{d_2-1}}{(m-1)!(d_2-m)!\lambda^{d_2-1}}.
\end{split}
\end{align}
Along this line, one will apply $e^{-q^\gamma/z}$ on the elementary fractions. We would like to note that in the middle of the computation, it will involve the following elementary lemma.
\begin{lem}
\[ 
\sum\limits_{i=0}^{a} (-1)^i {a \choose i}{b-i\choose c}={b-a \choose b-c}. 
\]
\end{lem}
There are various proofs of this lemma. For example, one can use the standard generating function technique in combinatorics. We omit the details. The rest of the computation is routine. 
\end{proof}

In the appendix, we will calculate the $F_{K,d,g}$, which uses only $t_{Q_0}(q^{\lin'}=0)$.

\begin{cor}
\begin{align}
\begin{split}
t_{Q_0}(q^{\lin'}=0)=\sum\limits_{m=1}^\infty &\left( \dfrac{- m^{m-2}q^{m\gamma}}{(m-1)!\lambda^{m-2}(\lambda+mz)} - 5H\cdot \dfrac{5m^{m-2}q^{m\gamma}}{(m-1)!\lambda^{m-2}(\lambda+mz)^2} \right.\\
&\left. - \dfrac{5H}{\lambda+mz}\cdot \left( \dfrac{5m^{m-2}(mH_m-m-2)q^{m\gamma}}{(m-1)!\lambda^{m-1}} +\dfrac{g_m(mq^\gamma/\lambda)}{\lambda} \right) \right),
\end{split}
\end{align}
where $g_m(z)=\sum\limits_{i>m}\dfrac{z^i}{i!(i-m)}$.
\end{cor}

\bigskip\bigskip

\appendix

\section{On the degree bound}
\centerline{\small by H.~Fan, Y.-P.~Lee and E.~Schulte-Geers}
\maketitle


\bigskip\bigskip

As mentioned in the Introduction, our results on the quantum Lefschetz are subjected to a degree bound \eqref{e:degbound}. In this appendix, we give strong evidences that the degree bound cannot be improved using the method discussed in this paper. Curiously, using a different approach, the authors in \cite{GJR} also arrived at the same range of indeterminacy in the special case of Calabi--Yau threefolds.

In this appendix, we discuss the special case of $X=\Pp^4$ and $D=Q$ the quintic hypersurface. 
In our proof of Theorem~\ref{theorem:recursion}, we uses only curve classes $\ell \in \operatorname{NE}(D) \to \operatorname{NE}(X)$, and the resulting degree bound is of the form $2g-2 < (\ell. D)$.
It turns out that the same degree bound remains even if a general curve class $\beta \in \operatorname{NE}(X)$ is used.
This conclusion will be discussed in Section~\ref{s:A.3}.

We use the following notation.
When discussing Gromov--Witten invariants of $Q$ or $\Pp^4$, since their Mori cones are one-dimensional, we only write a curve class as an integer $d$ to indicate $d$ times of their generators. For example, $\langle \dotsb \rangle^Q_{g,n,d}$.

Let's shift our attention back to equation \eqref{eqn:comp}. In quintic case, we only need to consider
\[
	\langle ~ \rangle_{g,0,d\lin}^\X=\displaystyle\int_{[\bM_{g,0}(\X,d\lin)]^{vir}} 1.
\]
If $5d < 2g-2$, Lemma \ref{lem:vanishing} fails. It appears to be remedied by considering the curve class $d\lin+K\gamma$ instead of $d\lin$ for a sufficiently large integer $K$. By a standard virtual dimension computation, we have
\begin{lem}
$vdim(\bM_{g,0}(\X,d\lin+K\gamma))=5d+2-2g+K$.
\end{lem}
Therefore when $5d+2-2g+K>0$, 
\[
	\displaystyle\int_{[\bM_{g,0}(\X,d\lin+K\gamma)]^{vir}} 1=0.
\]
Similar as before, we are left to analyze the graph sum. The leading term is the sum of graphs where there is a vertex over the quintic having the highest possible degree $d$. 
\begin{defn}
Let $\mathcal G$ be the set of degree $d\lin+K\gamma$, genus $g$, $0$-marked graphs with a vertex of degree $d$ genus $g$ over $Q_0$.
\end{defn} 

Similar to equation \eqref{eqn:vloc}, we have
\begin{align}\label{eqn:loc2}
0= \langle ~ \rangle^{\X}_{g,0,d\lin+K\gamma}
= \sum\limits_{\Gamma\in \mathcal G} Cont_\Gamma + \sum\limits_{\Gamma \not\in \mathcal G} Cont_{\Gamma}.
\end{align}
Without essential differences, the $Cont_{\Gamma}$ can be made explicit just like the ones in equation \eqref{eqn:loc}. Therefore, we won't repeat the lengthy descriptions here. Note that when $K=0$, $\sum\limits_{\Gamma\in \mathcal G} Cont_\Gamma = \langle ~ \rangle^{Q,\sO^-\oplus\sO(5)^+}_{g,0,d\lin}$, thus going back to the situation of \eqref{eqn:loc}.

Notice graphs in $\mathcal G$ has the following properties:
\begin{enumerate}
\item Except the degree $d$ genus $g$ vertex over $Q_0$, other vertices (if any) are of degree $0$ genus $0$;
\item As a result there is no loop in the underlying graph.
\end{enumerate}
One can evaluate those tree contributions coming out of the degree $d$ genus $g$ vertex over $Q_0$. By combining with Proposition \ref{prop:Qvertex}, $\sum\limits_{\Gamma\in \mathcal G} Cont_\Gamma$ can be calculated explicitly. In particular, it will be of the form $F_{K,d,g}\langle ~ \rangle_{g,0,d}^Q \lambda^{-2(1-g)+5d-K}$. The rest of the Appendix focuses on explicitly computing the coefficients $F_{K,d,g}$.

To bypass some of the combinatorics, we will use the $J$-function of $\X$ to extract the localization contribution of rational trees attached to the quintic vertex (See Section \ref{section:J}). But before that, we need to establish a formula for multiple point functions of $Q$.

\subsection{Computing $F_{K,d,g}$}
Now the leading term $F_{K,d,g}\langle ~ \rangle_{g,0,d}^Q \lambda^{-2(1-g)-5d-K}$ can be characterized as the $q^{K\gamma}$ coefficient of
\begin{equation}\label{e:FKg}
	\sum\limits_{n\geq 0} \dfrac{1}{n!} \langle t_{Q_0}(-\psi),\dotsc,t_{Q_0}(-\psi) \rangle_{g,n,d}^Q \lambda^{-2(1-g)-5d}.
\end{equation}

Recall that by string/dilaton/divisor equations, $\sum\limits_{n\geq 0} \dfrac{1}{n!} \langle t_{Q_0}(-\psi),\dotsc,t_{Q_0}(-\psi) \rangle_{g,n,d}^Q$ can be simplified into the form $F_{d,g}(q^{\lin'},q^\gamma)\langle ~ \rangle_{g,0,d}^Q\lambda^{-K}$ where $F_{d,g}(q^{\lin'},q^\gamma)$ is a power series in $q^{\lin'}, q^\gamma$.
\begin{defn}
$F_{d,g}(q^{\lin'},q^\gamma)$ is a power series in $q^{\lin'}, q^\gamma$ defined by the following equation.
\[
\sum\limits_{n\geq 0} \dfrac{1}{n!} \langle t_{Q_0}(-\psi),\dotsc,t_{Q_0}(-\psi) \rangle_{g,n,d}^Q = F_{d,g}(q^{\lin'},q^\gamma)\langle ~ \rangle_{g,0,d}^Q\lambda^{-K}
\]
\end{defn}

To simplify notations, recall we write $N_{g,d}=\langle ~ \rangle_{g,0,d}^Q$.
\begin{lem}
\[
\begin{split}
  &\langle \, \frac{s_1 1+ t_1 H}{1 - z_1\psi}, \dotsc, \frac{s_n 1+t_n H}{1- z_n \psi} \, \rangle^Q_{g,n,d} \\
= &N_{g,d} \sum_{I\subset(n), p\leq |I|} s_It_{\bar I}d^{|\bar I|} \left[\prod\limits_{r=1}^p (2g+|\bar I|+r-4)\right] \, \cdot \,  \sigma_p^I(z) \left( \sum_{i=1}^n z_i \right)^{|I|-p} ,
\end{split}
\]
where $(n)=\{1,2,\dotsc,n\}$, $\bar I=(n) - I$, $s_I=\prod\limits_{i\in I}s_i$ ($t_{\bar I}$ defined similarly), $\sigma^I_p(z)$ is the elementary symmetric polynomial of degree $p$ only using variables $z_i, i\in I$.
\end{lem}
It's a direct consequence of Proposition \ref{prop:Qvertex}. Our goal is a formula for the following.
\[  
\langle \, \frac{s_1 1+ t_1 H}{1 - z_1\psi} + \frac{u_1 H}{(1 - z_1\psi)^2}, \dotsc , \frac{s_n 1+t_n H}{1- z_n \psi} + \frac{u_nH}{(1 - z_n\psi)^2} \, \rangle^Q_{g,n,d}.
\]
A direct expansion shows that the expression can be written as
\[
\sum\limits_{J\subset (n)} \langle \, {\prod\limits_{k\in \bar J}}^\circ \, \dfrac{s_k 1 + t_k H}{1 - z_k\psi} \,,\, {\prod\limits_{j\in J}}^\circ \, \dfrac{u_j H}{(1 - z_j\psi)^2} \, \rangle^Q_{g,n,d},
\]
where $\prod^\circ$ means that different factors are treated as being in different insertions. Notice that applying $1+z_k\dfrac{\partial}{\partial z_k}$ turns the denominator of the $k$-th insertion into its square. A few further steps could go as follows.
\[
\begin{split}
& N_{g,d} \sum\limits_{J\subset (n)} \left[ \prod\limits_{j\in J} \left( 1+z_j\dfrac{\partial}{\partial z_j} \right) \right] \cdot \\
&\left. \qquad \cdot \left( \sum\limits_{I\subset (n), p\leq |I|} s_It_{\bar I}d^{|\bar I|} \left[\prod\limits_{r=1}^p(2g+|\bar I|+r-4)\right] \, \cdot \, \sigma_p^I(z) (\sum\limits_{i=1}^n z_i)^{|I|-p} \right) \right |_{s_j=0, t_j=u_j, j\in J}\\
=& N_{g,d} \sum\limits_{I\subset (n), p\leq |I|}  \left[\prod\limits_{r=1}^p(2g+|\bar I|+r-4)\right] \, \cdot \, \\
& \sum\limits_{J\subset \bar I, t\leq |J|} \left( \left[ \prod\limits_{r=0}^{t-1} (|I|-p-r) \right] \,\cdot\, \sigma_{p}^{I}(z)\sigma_t^J(z) (\sum\limits_{i=1}^n z_i)^{|I|-p-t} \,\cdot\, s_It_{\bar I-J}u_J d^{|\bar I|}\right)\\
\end{split}
\]
The equality is because the substitution $s_j=0, j\in J$ kills all the summands when $I\bigcap J\neq\emptyset$. Now a substitution is enough to determine $F_{d,g}(q^{\lin'},q^\gamma)$. 
\begin{defn}
\[
\bar A_1(m)=\sum\limits_{d_1\geq 0,d_2\geq m} q^{d_1\lin'+d_2\gamma} A_1(d_1,d_2,m),
\qquad
\bar A_2(m)=\sum\limits_{d_1\geq 0,d_2\geq m} q^{d_1\lin'+d_2\gamma} A_2(d_1,d_2,m).
\]
\end{defn}
We have the following huge unsimplified expression for $F_{d,g}(q^{\lin'},q^\gamma)$.
\begin{equation}
\begin{split}
F_{d,g}(q^{\lin'},q^\gamma)=& \sum\limits_{n=1}^\infty \dfrac{1}{n!} \sum\limits_{ k_1,\dotsb,k_n>0 } \sum\limits_{I\subset (n)} \sum\limits_{p=0}^{|I|}  d^{|\bar I|}\left\{ \left[\prod\limits_{r=1}^p(2g+|\bar I|+r-4)\right] \, \cdot \,  \right.\\
& \sum\limits_{J\subset \bar I} \sum\limits_{t=0}^{|J|} \left( \left[ \prod\limits_{r=0}^{t-1} (|I|-p-r) \right] \,\cdot\, \sigma_p^I(k_1,\dotsc,k_n) \sigma_t^J(k_1,\dotsc,k_n) (\sum\limits_{i=1}^n k_i)^{|I|-p-t} \,\cdot\, \right.\\
& \left. \left. \prod\limits_{i\in I} \bar A_1(k_i) \prod\limits_{l\in \bar I-J} 5(-\bar A_2(k_l)) \prod\limits_{j\in J} 5(-\bar A_1(k_j)) \right) \right\}.
\end{split}
\end{equation}

In order to simplify it, we first observe that the summation over $k_i$ can be rearranged, as the $k_i$ only appears after the elementary symmetric polynomials. Secondly, notice that the sets $I,J$ are only used as indices for the partitions $\{k_i\}$. Since $\{k_i\}$ ranges over all partitions of $K$, the sum is invariant under the symmetric group $\mathfrak S_n$ action on the indices of $\{k_i\}$. Therefore we conclude that only the sizes of the disjoint subsets $I,J\subset (n)$ matter in this sum.

\begin{equation} \label{e:19}
\begin{split}
F_{d,g}(q^{\lin'},q^\gamma)=& \sum_{n=1}^\infty \sum\limits_{\begin{subarray}{c} a+b \leq n \\ a,b\geq 0 \end{subarray}} \dfrac{1}{a! b! (n-a-b)!} (-5d)^{n-a} \, \cdot \\
&\left\{  \sum\limits_{p=0}^{a} \sum\limits_{t=0}^{b}  \prod\limits_{r=1}^p(2g+n-a+r-4) \, \cdot \, \prod\limits_{r=0}^{t-1} (a-p-r) \, \cdot \, \right.\\
& \left( \sum\limits_{ k_1,\dotsb,k_{n}>0 } (\sum\limits_{i=1}^n k_i)^{a-p-t} \sigma_p(k_1,\dotsc,k_{a}) \sigma_t(k_{a+1},\dotsc,k_{a+b}) \,\cdot\, \right.\\
& \left. \left. \prod\limits_{i=1}^a \bar A_1(k_i)
\prod\limits_{j=a+1}^{a+b} \bar A_1(k_j)
\prod\limits_{l=a+b+1}^{n} \bar A_2(k_l) \right) 
\right\}.
\end{split}
\end{equation}
Here the relationship with the previous summation is that $a=|I|, b=|J|$. Also note that we make some rearrangement of factors (e.g. pull out the factors $5$ from the third line and combine them into $(5d)^{n-a}$ in the first line). We can further combine the elementary symmetric polynomials.

\begin{equation}\label{e:20}
\begin{split}
F_{d,g}(q^{\lin'},q^\gamma)=& \sum_{n=1}^\infty \sum\limits_{\begin{subarray}{c} a+b \leq n \\ a,b\geq 0 \end{subarray}} \dfrac{1}{a! b! (n-a-b)!} (-5d)^{n-a} \, \cdot \\
&\left\{ \sum\limits_{p=0}^{a} \sum\limits_{t=0}^{b} \prod\limits_{r=1}^p(2g+n-a+r-4) \, \cdot \, \prod\limits_{r=0}^{t-1} (a-p-r)  \, \cdot \, \right.\\
&\left. \left( \dfrac{ {a\choose p}{b\choose t} }{ {a+b\choose p+t} } \sum\limits_{ k_1,\dotsc,k_{n}>0 } (\sum\limits_{i=1}^n k_i)^{a-p-t} \sigma_{p+t}(k_1,\dotsc,k_{a+b}) \,\cdot\, \prod\limits_{i=1}^{a+b} \bar A_1(k_i)
\prod\limits_{l=a+b+1}^{n} \bar A_2(k_l) \right) 
\right\}.
\end{split}
\end{equation}

In particular when $q^{\lin'}=0$,
\[\bar A_1(m)=-\dfrac{m^{m-2}q^{m\gamma}}{(m-1)!}, \qquad
\bar A_2(m)=-\dfrac{m^{m-2}(m-2+mH_m)q^{m\gamma}}{(m-1)!}+g_m(mq^\gamma),\]
where $g_m(z)=\sum\limits_{i>m}\dfrac{z^i}{i!(i-m)}$. For the convenience of later discussion, we need to set up a notation.

\subsection{Simplifying the generating series $F_{d,g}(z)$}\label{section:genfun}
In this subsection, Lagrange inversion will be heavily used. A good reference of related techniques could be, for example, \cite{Stanley}. From now on, we will let $z=q^\gamma$. Let $K$ be a positive integer and consider the sum ($m=2g-4$)

\begin{align*}
F_{d,g}(z)=&\sum_{n=1}^\infty \sum_{{a+b\leq n}\atop{a,b\geq 0}}\frac{(-1)^n}{a!b!(n-a-b)!}(-5d)^{n-a}\cdot\\
         &\bigg\{\sum_{p=0}^a\sum_{t=0}^b \prod_{r=1}^p(m+n-a+r)\cdot\prod_{r=0}^{t-1}(a-p-r)
         \bigg(\tfrac{{a \choose p}{b \choose t}}{{a+b \choose p+t}}\sum_{k_1,\ldots,k_n>0 }\big(\sum_{i=1}^nk_i\big)^{a-p-t}\cdot\\
         & \sigma_{p+t}(k_1,\ldots,k_{a+b})\prod_{i=1}^{a+b}\frac{k_i^{k_i-1}
z^i}{k_i!}\prod_{l=a+b+1}^n\bigg[\frac{(k_l-2+k_lH_{k_l})k_l^{k_l-1}z^{k_l}}{k_l!} -g_{k_l}(k_lz)\bigg]\bigg)\bigg\}
\end{align*}
Here $H_k$ denotes the harmonic sum $H_k=\sum_{j=1}^k\frac{1}{j}$ and 
$g_k(z)=\sum_{d>k}\frac{z^d}{d!(d-k)}$.

We use formal power series and begin with a straightforward observation. Let  
\begin{align}
\begin{split}
S_{a+b,p+t}(K):=&\sum_{{k_1+\ldots +k_n=K \atop  k_i\geq 1}} \sigma_{p+t}(k_1,\ldots,k_{a+b}) \\
&\prod_{i=1}^{a+b}\frac{k_i^{k_i-1}z^i}{k_i!}
\prod_{l=a+b+1}^n\big(\frac{(k_l-2k_l+k_lH_{k_l})k_l^{k_l-1}z^{k_l}}{k_l!}-g_{k_l}(k_lz)\big).
\end{split}
\end{align}
Let
\[
T(z)=\sum_{n\geq 1}\frac{n^{n-1}}{n!}z^n.
\]
It is well known that $T(x)$ (the ``tree function'') is the formal power series satisfying $T(x)=ze^{T(x)}$, and that for
a formal power series  $F$ the coefficients of $G(x):=F(T(x))$ are given by (Lagrange inversion)
\[[x^0]G(x)=[x^0] F(x), \qquad [x^k]G(x)=\dfrac{1}{k} [y^{k-1}] F^\prime(y)e^{ky} =[y^k](1-y)F(y)e^{ky}\text{ for } k\geq 1.\]
As a result, one easily sees that
\[\dfrac{T(x)}{1-T(x)}=\sum\limits_{i\geq 1}\dfrac{i^{i}}{i!}x^i.\]
Using these ingredients, one can rewrite $S_{a+b,p+t}(K)$ as
\[\displaystyle
S_{a+b,p+t}(K)={a+b \choose p+t} [u^K] \frac{T(zu)^{a+b}}{(1-T(zu))^{p+t}}\, Q(z,u)^{n-(a+b)},
\] where
\[
Q(z,u)=\sum_{k\geq 1} \left(\frac{(k-2+kH_{k})k^{k-1}z^{k}}{k!}-g_{k}(kz)\right)u^k.
\]

Now let $i\geq 0$ be fixed and consider the coefficient of $(-5d)^i$ (that is: $a=n-i$) in the power series $F_{d,g}(z)$ under the condition $k_1+\ldots +k_n=K$. Define
\begin{align*}
f_{i,K}=&\sum_{n\geq 1}\sum_{{n-i+b\leq n}\atop{n-i,b\geq 0}}\frac{(-1)^n}{(n-i)!b!(i-b)!}\cdot\\
         &\bigg\{\sum_{p=0}^{n-i}\sum_{t=0}^b K^{n-i-p-t}\prod_{r=1}^p(m+i+r)\cdot\prod_{r=0}^{t-1}(n-i-p-r)
         \bigg(\tfrac{{n-i \choose p}{b \choose t}}{{n-i+b \choose p+t}} S_{n-i+b, p+t}(K)\bigg)\bigg\}.
\end{align*}
Clearly only $n$ with $n-i-p-t\geq 0$ need to be considered.  
Using that, the observation above and changing the order of summation and cancelling/rewriting gives 
{\small 
\allowdisplaybreaks
\begin{align*}
f_{i,K}=&\sum_{b=0}^i  \frac{1}{b!(i-b)!}\bigg[\sum_{p\geq 0}\sum_{t=0}^b{m+i+p \choose p} {b\choose t} \cdot
\bigg\{\sum_{n\geq i+p+t} \frac{(-1)^n K^{n-i-p-t}}{(n-i-p-t)!}\,{n-i+b \choose p+t}^{-1} S_{n-i+b,p+t}(K)\bigg\}\bigg]\\
=&[u^K]\,\frac{1}{i!} \sum_{b=0}^i {i \choose b}\bigg[\sum_{p\geq 0} \sum_{t=0}^b{m+i+p \choose p} {b\choose t} 
\bigg\{\sum_{n\geq i+p+t} \frac{(-1)^nK^{n-i-p-t}}{(n-i-p-t)!}\frac{T(zu)^{n-i+b}}{(1-T(zu))^{p+t}}Q(z,u)^{i-b}\bigg\}\bigg]\\
=&[u^K]\,\frac{1}{i!} \sum_{b=0}^i {i \choose b}\bigg[\sum_{p\geq 0} \sum_{t=0}^b{m+i+p \choose p} {b\choose t} 
\bigg\{(-1)^{i+p+t} e^{-K T(uz)}\frac{T(zu)^{b+p+t}}{(1-T(zu))^{p+t}}Q(z,u)^{i-b}\bigg\}\bigg]\\
=&[u^K]\,\frac{1}{i!} \sum_{b=0}^i {i \choose b}\bigg[\sum_{p\geq 0} {m+i+p \choose p}  
\bigg\{(-1)^{i+p} e^{-K T(uz)}\big(1-\frac{T(zu)}{1-T(zu)}\big)^b\frac{T(zu)^{b+p}}{(1-T(zu))^{p}}Q(z,u)^{i-b}\bigg\}\bigg]\\
=&[u^K]\,\frac{1}{i!} (1-T(uz))^{m+i+1}\sum_{b=0}^i {i \choose b}\bigg[   
 \bigg\{(-1)^{i} e^{-K T(uz)}\big(1-\frac{T(zu)}{1-T(zu)}\big)^b T(zu)^{b} Q(z,u)^{i-b}\bigg\}\bigg]\\
=&[u^K]\,\frac{(-1)^i}{i!} (1-T(uz))^{m+i+1}   
 e^{-K T(uz)}\bigg(T(zu)\big(1-\frac{T(zu)}{1-T(zu)}\big)+ Q(z,u)\bigg)^i\\
\end{align*}
}
Now note that 
\[Q(z,u)=\frac{T(zu)}{1-T(zu)}-2T(zu) +G(z,u),\]
where 
\[
G(z,u)=\sum_{k\geq 1} \big(\dfrac{H_{k} k^{k}z^{k}}{k!}-g_{k}(kz)\big)u^k.
\]
Summing over $i$ gives
\begin{align*}\sum_{i\geq 0}(-5d)^i f_{i,K}   
&=[u^K] (1-T(uz))^{m+1}   e^{-K T(uz) +5d (1-T(zu))G(z,u)} \\
&= z^K [u^K] (1-T(u))^{m+1}   e^{-K T(u) +5d(1-T(u)) G(z,u/z)} \\
&= z^K [y^K] (1-y)^{m+2}   e^{5d (1-z)G(z,y e^{-y}/z)}\\
\end{align*}
Finally summing  over $K\geq 0$ gives
\[
F_{d,g}(z)= (1-z)^{m+2} e^{5d(1-z) G(z,\,e^{-z})}.
\]
To simplify further, one needs the following lemma. 
\begin{lem}
\[(1-z) G(z,e^{-z})=-\log(1-z).\]
\end{lem}
\begin{proof} 
We have to show that for $m\geq 1$ the coefficient $g_m:=[z^m]G(z,e^{-z})=H_m$. We have
$$g_m=\frac{1}{m!}\sum_{k=1}^m (-1)^{m-k}{m \choose k} \sum_{j=1}^k\frac{1}{j}\bigg(k^m -(k-j)^m\bigg)$$
Rewrite $\frac{1}{j}\left(k^m-(k-j)^m\right)=\sum_{i=1}^m (-1)^{i+1}{m \choose i} j^{i-1}k^{m-i}$, then

\begin{align*}
g_m&=\frac{1}{m!}\sum_{k=1}^m (-1)^{m-k}{m \choose k} \sum_{j=1}^k\frac{1}{j}\bigg(k^m -(k-j)^m\bigg)\\
               &=\sum_{i=1}^{m}(-1)^{i+1}{m \choose i}\bigg[\frac{1}{m!}\sum_{k=1}^m (-1)^{m-k}{m \choose k}\bigg(\sum_{j=1}^k j^{i-1}k^{m-i}\bigg)\bigg]\\
               &=\sum_{i=1}^{m}(-1)^{i+1}{m \choose i}\frac{1}{i}       =H_m
\end{align*}

\end{proof}

Here the equality from the second to the third line follows from the following facts.
\begin{enumerate}
\item $\bigg(\sum_{j=1}^k j^{i-1}k^{m-i}\bigg)$ is polynomial of degree $m$ in $k$,
with leading term $\frac{k^m}{i}$;
\item the alternating sum $\Delta_{f,m}:=\sum_{k=0}^m (-1)^{m-k}{m \choose k} f(k)$ is the $m$-th forward difference of $f$ in $0$;
\item if $f$ is polynomial of degree $\leq m$ then $\Delta_{f,m}=m!\,[x^m] f(x)$.
\end{enumerate}
 
It follows that the final result is extremely simple:
\begin{thm}\label{theorem:genfun}
\[F_{d,g}(z)=(1-z)^{2g-2-5d}.\]
\end{thm}

As a brief side note, we would like to discuss the general $F_{d,g}(q^{\lin'},q^\gamma)$ without $q^\gamma$ being set to $0$. Let $A(u):=-\sum_{m\geq 1} \bar{A}_1(m)u^m$ 
and $B(u):=-\sum_{m\geq 1} \bar{A}_2(m)u^m$. The same computation as above
gives
\[
F_{d,g} =\sum_{K\geq 0} [u^K]  e^{-KA(u)} (1+uA^\prime(u))^{-(m+1)} e^{5d \left(A(u)(1-uA^\prime(u))+B(u)\right)}.
\]
Let $h(u)$ denote the compositional inverse of the formal power series $u e^{A(u)}$, then
this can be rewritten as:
\[F_{d,g} = \left(1+h(1)A^\prime(h(1))\right)^{-(m+2)} \exp\bigg(5d\big(A(h(1))(1-h(1)A^\prime(h(1)))+B(h(1)\big)\bigg).
\]

\subsection{Conclusion} \label{s:A.3}
Theorem~\ref{theorem:genfun} implies that whenever $2-2g+5d\leq 0$, if $2-2g+5d+K>0$, $F_{K,d,g}$ is always $0$. Thus by adding multiples of $\gamma$, we do not gain new understanding beyond the degree bound. 

What if we use a general curve class $d_1 \lin + K \gamma + d_3 \gamma'$?
Let us define the number $F_{K, d_3, d, g}$ to be the corresponding coefficient in \eqref{e:FKg}, with the curve class $d \lin + K \gamma$ in \eqref{eqn:loc2} replaced by $d_1 \lin + K \gamma + d_3 \gamma'$.
Extensive numerical checks suggest that
\[
 F_{K,d_3,d_1,g}=(-1)^{d_3} F_{K+d_3,d_1,g} {K+d_3 \choose K}.
\]
As a result, the same degree bound remains.

\bibliographystyle{amsxport}
\bibliography{ref}

\end{document}